\documentclass[letter,11pt]{article}

\usepackage[linktocpage=true]{hyperref}
\usepackage[margin=1.2 in, top=1 in, bottom= 1.2 in]{geometry}

\usepackage[normalem]{ulem}

\makeatletter
\g@addto@macro\normalsize{
  \setlength\abovedisplayskip{8pt}
  \setlength\belowdisplayskip{8pt}
  \setlength\abovedisplayshortskip{8pt}
  \setlength\belowdisplayshortskip{8pt}
  }
\makeatother

\usepackage{tocloft}

\usepackage[english]{babel}
\usepackage{amsfonts}
\usepackage{mathrsfs}
\usepackage{bbm}
\usepackage{latexsym}
\usepackage{math dots}
\usepackage{amssymb}
\usepackage{mathtools}

\usepackage{enumitem}
\setlist{nolistsep}
\usepackage{amsthm}
\usepackage[capitalize]{cleveref}
\crefformat{footnote}{#2\footnotemark[#1]#3} 

\newcommand\eqnitem[1][]{%
  \ifx\relax#1\relax  \item \else \item[#1] \fi
  \abovedisplayskip=0pt\abovedisplayshortskip=0pt~\vspace*{-\baselineskip}}

\newtheoremstyle{plain}{3mm}{3mm}{\slshape}{}{\bfseries}{.}{.5em}{}
\newtheoremstyle{definition}{2mm}{2mm}{}{}{\bfseries}{.}{.5em}{}
\theoremstyle{plain}
	
\newtheorem{theorem}{Theorem}

\newtheorem{lemma}[theorem]{Lemma}

\newtheorem{corollary}[theorem]{Corollary}

\newtheorem{claim}[theorem]{Claim}

\theoremstyle{definition}
\newtheorem{definition}[theorem]{Definition}
\newtheorem{remark}[theorem]{Remark}

\theoremstyle{plain}
\newcounter{MainTheoremCounter}

\newtheorem{Maintheorem}[MainTheoremCounter]{Theorem}

\theoremstyle{plain}
\newtheorem*{namedthm}{\namedthmname}
\newcounter{namedthm}
\makeatletter
	\newenvironment{named}[2]
	{\def\namedthmname{#1}
	\refstepcounter{namedthm}
	\namedthm[#2]\def\@currentlabel{#1}}
	{\endnamedthm}
\makeatother

\usepackage{chngcntr}
\counterwithin{theorem}{section}

\numberwithin{equation}{section}

\allowdisplaybreaks

\usepackage{xcolor}
\definecolor{Color2}{rgb}{0.78, 0.11, 0.0}
\hypersetup{citecolor = black,colorlinks,
			linkcolor = black,
			urlcolor = Color2}

\usepackage{titlesec}
\titleformat{\section}
  {\Large\center\bfseries}
  {\thesection.}{.7em}{}
\titlespacing*{\section}{0pt}{3.5ex plus 0ex minus 0ex}{1.5ex plus 0ex}
\titleformat{\subsection}
  {\large\center\bfseries}
  {\thesubsection.}{.7em}{}
\titlespacing*{\subsection}{0pt}{3.5ex plus 0ex minus 0ex}{1.5ex plus 0ex}
\titleformat{\subsubsection}
  {\center\bfseries}
  {\thesubsubsection.}{.7em}{}
\titlespacing*{\subsubsection}{0pt}{3.5ex plus 0ex minus 0ex}{1.5ex plus 0ex}

\addto\captionsenglish{}

\usepackage{titling}
\newcommand{\supp}{{\normalfont\text{supp}}\,}

\newcommand{\eps}{\epsilon}
\newcommand{\N}{\mathbb{N}}
\newcommand{\Z}{\mathbb{Z}}
\newcommand{\R}{\mathbb{R}}

\newcommand{\Nz}{\N_0}
\newcommand{\defeq}{\vcentcolon=}

\newcommand\restr[2]{{ \left.\kern-\nulldelimiterspace #1 \right|_{#2}}}

\renewcommand{\epsilon}{\varepsilon}
\renewcommand{\leq}{\leqslant}
\renewcommand{\geq}{\geqslant}
\renewcommand{\setminus}{\backslash}

\usepackage[normalem]{ulem}

\newcommand{\al}{\alpha}

\newcommand{\ogamma}{\overline{\gamma}}
\newcommand{\gone}{\gamma_0}
\newcommand{\gtwo}{\gamma_1}
\newcommand{\gthree}{\gamma_2}
\newcommand{\gfour}{\gamma_3}
\newcommand{\gfive}{\gamma_4}

\newcommand{\NC}{\mathcal{N}}
\newcommand{\QC}{\mathcal{Q}}
\newcommand{\HC}{\mathcal{H}}
\newcommand{\QCt}{\widetilde{\mathcal{Q}}}

\newcommand{\AC}{\mathcal{A}}
\newcommand{\JC}{\mathcal{J}}
\newcommand{\CC}{\mathcal{C}}

\newcommand{\FC}{\mathcal{F}}
\newcommand{\dimh}{\dim_{\text{H}}}

\newcommand{\dimm}{\dim_{\text{M}}}
\newcommand{\udimm}{\overline{\dim}_{\text{M}}\hspace*{.1em}}
\newcommand{\ldimm}{\underline{\dim}_{\text{M}}\hspace*{.1em}}

\newcommand{\Leb}{\textnormal{Leb}}

\renewcommand{\Phi}{\mathfrak{R}}
\renewcommand{\Psi}{\mathfrak{L}}

\newcommand{\inftycontent}{>0}

\begin{document}

\title{\bfseries A combinatorial proof of a sumset conjecture of Furstenberg}
\author{Daniel Glasscock \and Joel Moreira \and Florian K.\ Richter}

\date{\small \today}
\maketitle

\begin{abstract}
We give a new proof of a sumset conjecture of Furstenberg that was first proved by Hochman and Shmerkin in 2012: if $\log r / \log s$ is irrational and $X$ and $Y$ are $\times r$- and $\times s$-invariant subsets of $[0,1]$, respectively, then $\dimh (X+Y) = \min ( 1, \dimh X + \dimh Y)$. Our main result yields information on the size of the sumset $\lambda X + \eta Y$ uniformly across a compact set of parameters at fixed scales.  The proof is combinatorial and avoids the machinery of local entropy averages and CP-processes, relying instead on a quantitative, discrete Marstrand projection theorem and a subtree regularity theorem that may be of independent interest.
\end{abstract}

\small
\tableofcontents
\thispagestyle{empty}
\normalsize


\section{Introduction}

Given $r\in\N$, a set $X \subseteq [0,1]$ is \emph{$\times r$-invariant} if it is closed and $T_r X \subseteq X$, where $T_r\colon[0,1]\to[0,1]$ is the map $x \mapsto rx \pmod 1$.  In the late 1960's, Furstenberg conjectured\footnote{To the authors' knowledge, the sumset conjecture \eqref{eqn_furstenberg_sumset_conjecture} does not appear in print but was known to have originated with Furstenberg. The intersection conjecture \eqref{eqn_furstenberg_intersection_conjecture}, on the other hand, is one of several conjectures stated in  \cite{furstenbergtransversality}.} that if $r$ and $s$ are multiplicatively independent positive integers (that is, $\log r / \log s$ is irrational) and $X$ and $Y$ are $\times r$- and $\times s$-invariant, respectively, then
\begin{align}
\label{eqn_furstenberg_sumset_conjecture}
    \dimh \big(X+Y \big) &= \min \big(1, \dimh X + \dimh Y \big), \text{ and} \\
\label{eqn_furstenberg_intersection_conjecture}
    \dimh \big(X \cap Y \big) &\leq \max \big(0, \dimh X + \dimh Y - 1 \big). 
\end{align}
The sumset conjecture \eqref{eqn_furstenberg_sumset_conjecture} was resolved by Hochman and Shmerkin \cite{localentropy}, who proved a more general result concerning the dimension of sums of invariant measures.  It also follows by more recent results of Shmerkin \cite{shmerkin} and Wu \cite{wu}, who independently resolved a generalization of the intersection conjecture \eqref{eqn_furstenberg_intersection_conjecture}. We give a more detailed account of this recent history later in the introduction.

The purpose of this article is to give a new, combinatorial proof of Furstenberg's sumset conjecture \eqref{eqn_furstenberg_sumset_conjecture}.  Denoting the unlimited $\gamma$-Hausdorff content by $\HC_{\inftycontent}^\gamma$ (see \cref{def_unlmt_hausdorff_content}), our main theorem is as follows.

\begin{Maintheorem}
\label{main_real_theorem_intro}
Let $r$ and $s$ be multiplicatively independent positive integers, and let $X, Y \subseteq [0,1]$ be $\times r$- and $\times s$-invariant sets, respectively. Define $\ogamma = \min \big ( \dimh X + \dimh Y, 1 \big)$. For all compact $I \subseteq \R \setminus \{0\}$ and all $\gamma < \ogamma$,
\begin{align}
\label{eqn_intro_result}
    \inf_{\lambda, \eta \in I}\ \HC_{\inftycontent}^{\gamma} \big( \lambda X + \eta Y \big) > 0.
\end{align}
\end{Maintheorem}

Beyond implying \eqref{eqn_furstenberg_sumset_conjecture}, \cref{main_real_theorem_intro} gives finer quantitative information on the size of the sumset $\lambda X + \eta Y$ in terms of the unlimited $\gamma$-Hausdorff content uniformly over the parameters $\lambda$ and $\eta$.  
The uniformity in the result, which does not appear to follow from \cite{localentropy}, has found use in recent applications concerning digit problems; see, for example, \cite{GMR_add_trans_2020} and \cite{yu_digit_expansions_arxiv}.
See \cref{remark_quantifyingA} below for some further discussion on this uniformity.

Our proof of \eqref{eqn_furstenberg_sumset_conjecture} differs from other proofs in the literature in that it completely avoids the machinery of CP-processes and local entropy averages.  Instead, it features an elementary, combinatorial approach that builds on the work of Peres and Shmerkin in~\cite{peresshmerkin2009}.  Important ingredients in the proof include a quantitative discrete Marstrand theorem (\cref{thm:discretemarstrandbase}) and a subtree regularity theorem (\cref{tree_thinning_theorem}), both of which may be of independent interest.

\subsection{History and context}

In a highly influential work in geometric measure theory, Marstrand \cite{Marstrand_1954} related the Hausdorff dimension of a Borel set $E \subseteq \R^2$, $\dimh E$, to the Hausdorff dimension of its images under orthogonal projections and its intersections with lines.  More specifically, he showed that for almost every line $L \subseteq \R^2$, $\dimh( \pi_L E) = \min \big(1,\dimh E \big)$, where $\pi_L$ is the orthogonal projection $\R^2 \to L$, and that for almost every line $L$ intersecting $E$, $\dimh( E \cap L) = \max \big(0,\dimh E - 1 \big)$.\footnote{Marstrand considered sets of positive, finite Hausdorff measure.  His results imply the ones stated here when combined with the fact that any Borel set of Hausdorff dimension $\gamma > 0$ can be approximated from above and below by sets with positive, finite $(\gamma \pm \eps)$-Hausdorff measure.  For a more modern take on these theorems, see Corollary 9.4 and Theorem 10.10 in \cite{mattila_geometry_of_sets_1995}.}

Images of a Cartesian product $X \times Y$ under orthogonal projections are, up to affine transformations which preserve dimension, sumsets of the form $\lambda X +  \eta Y$, while intersections of $X \times Y$ with lines are affinely equivalent to sets of the form $\lambda X \cap (\eta Y + \sigma)$. Thus, Marstrand's theorems in the case $E=X\times Y$ imply the following.

\begin{theorem}[{\cite[Theorems II and III]{Marstrand_1954}}]
\label{theorem_marstrand}
Let $X$ and $Y$ be Borel subsets of $[0,1]$. For Lebesgue-a.e. $\lambda, \eta, \sigma \in \R$,
\begin{align}
    \label{eqn_theorem_marstrand}
    \dimh\big( \lambda X +  \eta Y \big) &= \min{}\big(\dimh (X \times Y) , \ 1 \big), \text{ and}\\
    \label{marstrand_slicing_result_intro}
    \dimh\big( \lambda X \cap (\eta Y + \sigma) \big) &\leq \max{}\big(0, \ \dimh (X \times Y) - 1 \big).
\end{align}
\end{theorem}

Improving \eqref{eqn_theorem_marstrand} and \eqref{marstrand_slicing_result_intro} by replacing the Lebesgue-typical projection or intersection of $X\times Y$ with a concrete projection or intersection is not possible in general \cite{Mattila_kaufman_example_1975} but can be done in special cases when the sets $X$ and $Y$ are structured.  Furstenberg's conjectures \eqref{eqn_furstenberg_sumset_conjecture} and \eqref{eqn_furstenberg_intersection_conjecture} can be contextualized as such: when $r$ and $s$ are multiplicatively independent and $ X \times Y$ is the product of a $\times r$- and a $\times s$-invariant set, results for the Lebesgue-typical projection and intersection should hold for the orthogonal projection to, and the intersection with, the line $x=y$.  These conjectures join a host of results and conjectures by Furstenberg and others that aim to capture the independence between base-$r$ and base-$s$ structure when $r$ and $s$ are multiplicatively independent.

Conjectures \eqref{eqn_furstenberg_sumset_conjecture} and \eqref{eqn_furstenberg_intersection_conjecture} were recently resolved, both proven in more general forms.  In the following theorem, we have combined special cases of the results by Hochman and Shmerkin \cite{localentropy}, Shmerkin \cite{shmerkin}, and Wu \cite{wu} that are most relevant to this work. Note that $\udimm$ denotes the upper Minkowski dimension (see \cref{def_metric_entropy_and_box_dims}).

\begin{theorem}[{\cite{localentropy} and \cite{shmerkin,wu}}]
\label{thm_HS_localentropy_and_intersections}
Let $r$ and $s$ be multiplicatively independent positive integers, and let $X, Y \subseteq [0,1]$ be $\times r$- and $\times s$-invariant sets, respectively. For all $\lambda, \eta \in \R \setminus \{0\}$ and all $\sigma \in \R$,
\begin{align}
\label{eqn_in_thm_HS_localentropy}
     \dimh\big( \lambda X +  \eta Y \big) &= \min{}\big(\dimh X + \dimh Y, \ 1 \big), \text{ and} \\
\label{eqn_furstenberg_intersection}
     \udimm\big( \lambda X \cap  (\eta Y + \sigma) \big) &\leq \max{} \big(0, \ \dimh X + \dimh Y - 1 \big).
\end{align}
\end{theorem}

A number of partial results preceded those in \cref{thm_HS_localentropy_and_intersections}, both for multiplicatively invariant sets and for attractors of iterated function systems (IFSs).  Carlos Moreira \cite{moreira_sums_of_regular_cantor_sets} considered sumsets of attractors of IFSs with certain irrationality and non-linearity conditions.  Peres and Shmerkin \cite{peresshmerkin2009} proved \eqref{eqn_in_thm_HS_localentropy} for attractors of IFSs with rationally independent contraction ratios; this resolved \eqref{eqn_in_thm_HS_localentropy} in the special case that $X$ and $Y$ are restricted digit Cantor sets with respect to multiplicatively independent bases. (This work of Peres and Shmerkin is particularly relevant to the arguments in this paper, as we explain in detail in \cref{section_outline_of_proof}.)

Hochman and Shmerkin \cite{localentropy} developed Furstenberg's CP processes \cite{furstenbergtransversality} and introduced local entropy averages to prove \eqref{eqn_in_thm_HS_localentropy} both for invariant sets and measures and for attractors of IFSs satisfying some general minimality conditions.  Wu \cite{wu} combined the CP process machinery with Sinai's factor theorem from ergodic theory to resolve \eqref{eqn_furstenberg_intersection} for invariant sets and attractors of regular, self-similar IFSs.  Shmerkin \cite{shmerkin} resolved \eqref{eqn_furstenberg_intersection} utilizing tools primarily from additive combinatorics, proving an inverse theorem for the decay of $L^q$ norms of certain self-similar measures of dynamical origin. Yu \cite{yu_improvement_to_furstenberg_arxiv} and Austin \cite{austin_proof_of_furstenberg_2020} gave dynamical proofs of \eqref{eqn_furstenberg_intersection_conjecture}, simplifying some aspects of earlier proofs.

The sumset and intersection theorems are closely related: fibers of orthogonal projections are precisely those lines with which intersections are considered.
It is not surprising, then, that the intersection theorem can be used to deduce the sumset theorem.
For example, if for arbitrary sets $X, Y \subseteq [0,1]$ we know that for all $\gamma > \max \big( 0, \dimh X + \dimh Y - 1 \big)$, there exists $\delta_0 > 0$, for all $0 < \delta < \delta_0$, and for all balls $B$ of diameter $\delta$,
\[\NC \big( X \cap (Y + B), \delta \big) \leq \delta^{-\gamma},\]
then we can deduce that $\dimh (X+Y) = \min (1, \dimh (X \times Y) \big)$. 
This type of uniformity is made explicit in Shmerkin \cite{shmerkin} and Yu \cite{yu_improvement_to_furstenberg_arxiv} and may be implicit in the other proofs of the intersection conjecture.  It is possible to deduce \cref{main_real_theorem_intro} from Shmerkin's main result in \cite{shmerkin}; we explain the details in the course of another argument in \cite{GMR_add_trans_2020}.  Despite the fact that every proof of the intersection conjecture can be counted as a proof of the sumset conjecture, we believe our approach still has merit: it is the most elementary proof to date; it exposes uniformity important in certain number-theoretic applications; and it features tools which may be of independent interest.

\cref{main_real_theorem_intro} has a geometric formulation in terms of orthogonal projections; while we will not make particular use of the theorem in this form, it is worth formulating for its historical connection to the topic. Let $\pi_\theta: \R^2 \to \R^2$ be the orthogonal projection onto the line that contains the origin and forms an angle $\theta$ with the positive $x$-axis.  The proof of the equivalence between \cref{main_real_theorem_intro} and \cref{theorem_uniform_hausdorff_dimension_projections_projection_form} is standard and not needed in this work, so it is omitted.

\begin{Maintheorem}
\label{theorem_uniform_hausdorff_dimension_projections_projection_form}
Let $r$ and $s$ be multiplicatively independent positive integers, and let $X, Y \subseteq [0,1]$ be $\times r$- and $\times s$-invariant sets, respectively. Define $\ogamma = \min \big ( \dimh X + \dimh Y, 1 \big)$. For all compact $I \subseteq (0,\pi) \setminus \{\pi / 2\}$ and all $\gamma < \ogamma$, $\inf_{\theta\in I} \HC_{\inftycontent}^{\gamma} \big( \pi_\theta (X \times Y) \big) > 0.$
\end{Maintheorem}

\subsection{Overview of the paper}

The paper is organized as follows. In \cref{sec_2}, we organize the terminology, notation, and basic facts we need from discrete and continuous fractal geometry, including some properties of $\times r$-invariant subsets of $[0,1]$ and an equidistribution lemma. \cref{section_projections_and_discrete_marstrand} contains a proof of \cref{thm:discretemarstrandbase}, our discrete Marstrand projection theorem. \cref{tree_section} features notation and terminology for trees and the subtree regularity theorem, \cref{tree_thinning_theorem}. Finally, we prove \cref{main_real_theorem_intro} in \cref{sec_proof_of_theorem}.

\subsection{Acknowledgements}

The authors extend their thanks to the referees for their valuable feedback.  The effectiveness of the main result was pointed out to us by the referees and led to \cref{remark_quantifyingA}. The third author is supported by the National Science Foundation under grant number DMS~1901453.

\section{Preliminary definitions and results}
\label{sec_2}

The positive and non-negative integers are denoted by $\N$ and $\Nz$, respectively. For $x \in \R$, denote the fractional part by $\{x\}$ and the integer part (or floor) by $\lfloor x \rfloor$.
The Lebesgue measure on the real line is denoted by $\Leb$. Throughout the paper, $\R^d$ is equipped with the Euclidean norm which we denote by $| \cdot |$. Given two positive-valued functions $f$ and $g$, we write $f \ll_{a_1,\ldots,a_k} g$ or $g \gg_{a_1,\ldots,a_k} f$ if there exists a constant $K > 0$, depending only on the quantities $a_1, \ldots, a_k$, for which $f(x) \leq K g(x)$ for all $x$ in the domain common to both $f$ and $g$.  We write $f \asymp_{a_1,\ldots,a_k} g$ if both $f \ll_{a_1,\ldots,a_k} g$ and $f \gg_{a_1,\ldots,a_k} g.$

\subsection{Continuous and discrete fractal geometry}
\label{sec_cont_discr_fractal_geom}

In this section, we lay out the notation, tools, and results we need from continuous and discrete fractal geometry.  A good general reference for the standard material in this section is \cite[Ch. 4]{mattila_geometry_of_sets_1995}.  In the definitions that follow, $\rho, \gamma, c > 0$, $d \in \N$, and $X \subseteq \R^d$ is non-empty. 

\begin{definition} \leavevmode
\label{def_metric_entropy_and_box_dims}
\begin{itemize}
    \item The set $X$ is \emph{$\rho$-separated} if for all distinct $x_1, x_2 \in X$, $|x_1 - x_2| \geq \rho$.
    \item The \emph{metric entropy of $X$ at scale $\rho$} is
\[\NC(X,{\rho}) = \sup \big \{ |X_0| \ \big | \ X_0 \subseteq X \text{ is } \rho \text{-separated} \big \}.\]
    \item The \emph{lower Minkowski dimension} of $X$ is
    \begin{equation}
    \ldimm X = \liminf_{\delta \to 0^+} \frac{\log \NC (X, \delta)}{\log \delta^{-1}}.
    \end{equation}
    The \emph{upper Minkowski dimension}, $\udimm X$, is defined analogously with a limit supremum in place of the limit infimum.  If $\ldimm X = \udimm X$, then this value is the \emph{Minkowski dimension} of $X$, $\dimm X$.
\end{itemize}
\end{definition}

It is a well-known fact which we will use without further mention that if $\rho < 1$, then $\ldimm X = \liminf_{N \to \infty} \log \NC (X, \rho^{N}) \big/ \log \rho^{-N}$ and $\udimm X = \limsup_{N \to \infty} \log \NC (X, \rho^{N}) \big/ \allowbreak \log \rho^{-N}$.

\begin{definition}\leavevmode
\label{def_unlmt_hausdorff_content}
\begin{itemize}
    \item The \emph{unlimited $\gamma$-Hausdorff content}\footnote{The unlimited $\gamma$-Hausdorff content is sometimes denoted in the literature by $\HC_\infty^\gamma$. We choose to use $\HC_{\inftycontent}^\gamma$ because it is more consistent with the notation introduced in \cref{def_rho_gamma_sets}.} of $X$ is
    \[\HC_{\inftycontent}^\gamma (X) = \inf \left\{ \sum_{i \in I} \delta_i^\gamma \ \middle| \ X \subseteq \bigcup_{i \in I} B_i, \ B_i \text{ open ball of diameter } \delta_i \right\}. \]
    Note that when $X$ is compact, the index set $I$ may be taken to be finite.
    \item The \emph{Hausdorff dimension} of $X$ is
    \begin{align*}
        \dimh X &= \sup \{ \gamma \in \R \ | \ \HC_{\inftycontent}^\gamma (X) > 0 \}\\
        &=\inf \{ \gamma \in \R \ | \ \HC_{\inftycontent}^\gamma (X) = 0 \}.
    \end{align*}
\end{itemize}
\end{definition}

In the following definition, we introduce two notions meant to capture the dimensionality of discrete sets.

\begin{definition}\leavevmode \label{def_rho_gamma_sets}
\begin{itemize}
    \item (cf. \cite[Definition 1.2]{katztao2001}) The set $X$ is a \emph{$(\rho,\gamma)_c$-set} if it is $\rho$-separated and for all $\delta \geq \rho$ and all open balls $B$ of diameter $\delta$,
\begin{align}\label{inequality_def_delta_alpha_set}
    \big| X \cap B \big| \leq c \left( \frac{\delta}{\rho} \right)^\gamma.
\end{align}
    \item The \emph{discrete Hausdorff content of $X$ at scale $\rho$ and dimension $\gamma$} is
    \[\HC_{\geq \rho}^\gamma (X) = \inf \left\{ \sum_{i\in I} \delta_i^\gamma \ \middle| \ X \subseteq \bigcup_{i\in I} B_i, \ B_i \text{ open ball of diameter } \delta_i \geq \rho \right\}.\]
    Note that when $X$ is compact, the index set $I$ may be taken to be finite.
\end{itemize}
\end{definition}

In the definition of a $(\rho,\gamma)_c$-set, we think of $\rho$ as being positive and close to $0$, $\gamma \in [0,d]$ as the ``dimension'' of the set, and $c > 0$ as an uninteresting parameter that exists only to make our arguments explicit. The inequality in \eqref{inequality_def_delta_alpha_set} guarantees that the points of a $(\rho,\gamma)_c$-set cannot be too concentrated in any ball. It follows from that inequality that the maximum cardinality of a $(\rho,\gamma)_c$ set in $[0,1]^d$ is on the order of $\rho^{-\gamma}$. A $(\rho,\gamma)_c$-set with cardinality $\gg \rho^{-\gamma}$ can be thought of as a discrete approximation to a set with Hausdorff dimension $\gamma$; this is made more precise in \cref{remark_on_rho_gamma_c_set_metric_entropy} below and is realized in \cref{discrete_invariant_set_bounds}. In fact, if the discrete approximations of a set $X \subseteq \R^d$ at all scales $\rho > 0$ are $(\rho,\gamma)_c$-sets, then the \emph{Assouad dimension} (cf. \cite[Section 2.1]{fraser_2020}) of the set $X$ is at most $\gamma$.  More precisely, the Assouad dimension of $X$ is the infimum of the set of $\gamma$'s for which there there exists $c > 0$ such that for all $\rho > 0$, the set $X$ rounded to the lattice $\rho \Z^d$ is a $(\rho,\gamma)_c$-set.

The discrete Hausdorff content at scale $\rho$ is a ``$\rho$-resolution'' analogue of the unlimited Hausdorff content.  The discrete Hausdorff contents of two sets that look the same at scale $\rho$ are approximately equal.  The following lemma provides a connection between the discrete and the continuous regimes that will be useful in the proof of \cref{main_real_theorem_intro}.

\begin{lemma}
\label{lemma_connection_between_discrete_and_unlimited_H_contents}
Let $X \subseteq \R^d$ be compact. For all $\gamma \geq 0$,
\begin{align}\label{equality_connection_between_discrete_and_unlimited_H_contents}
    \lim_{\rho \to 0^+} \HC_{\geq \rho}^\gamma (X) = \HC_{\inftycontent}^\gamma(X).
\end{align}
Consequently, if $\lim_{\rho \to 0} \HC_{\geq \rho}^\gamma (X)  > 0$, then $\dimh X \geq \gamma$.
\end{lemma}

\begin{proof}
Let $\gamma \geq 0$. The limit in \eqref{equality_connection_between_discrete_and_unlimited_H_contents} exists because the function $\rho \mapsto \HC_{\geq \rho}^\gamma (X)$ is non-increasing as $\rho$ tends to $0^+$ and is bounded from below by $\HC_{\inftycontent}^\gamma(X)$.  Equality in the limit follows from the fact that $X$ is compact, allowing for the index set in the definition of $\HC_{\inftycontent}^\gamma(X)$ to be taken to be finite. If $\lim_{\rho \to 0} \HC_{\geq \rho}^\gamma (X)  > 0$, then $\HC_{\inftycontent}^\gamma(X) > 0$, and it follows from the definition of the Hausdorff dimension that $\dimh X \geq \gamma$.
\end{proof}

\begin{remark}\label{remark_on_rho_gamma_c_set_metric_entropy}
It would be natural to define the \emph{metric entropy at scale $\rho$ and dimension $\gamma$} of the set $X$ as
\[\NC( X,{(\rho,\gamma)_c}) = \sup \big \{ |X_0| \ \big | \ X_0 \subseteq X \text{ is a } (\rho,\gamma)_c\text{-set} \big \}.\]
Using a max flow, min cut argument similar to the one in \cite[Ch. 3]{bishopperesbook}, it can be shown that for $X$ compact,
\begin{align}\label{eqn:covsameasn} \frac{\NC \big(X,{(\rho,\gamma)_c} \big)}{\rho^{-\gamma}} \asymp_{c,d} \HC_{\geq \rho}^\gamma(X).\end{align}
Thus, $(\rho,\gamma)_c$-sets of cardinality $\gg \rho^{-\gamma}$ can be thought of as discrete fractal sets of dimension $\gamma$. We will not need \eqref{eqn:covsameasn}; the interested reader can consult \cite[Prop. A1]{Fassler_Orponen_2014} for some details.
\end{remark}

The following is a discrete version of the well-known mass distribution principle, cf. \cite[Lemma 1.2.8]{bishopperesbook}.

\begin{lemma}\label{lem_consequence_of_discrete_frostman}
Let $\mu$ be a Borel probability measure on $\R^d$, and let $\rho, \kappa > 0$. If for all balls $B$ of diameter $\delta \geq \rho$, $\mu (B) \leq \kappa \delta^\gamma$, then the support $\supp\mu$ of $\mu$ satisfies $\HC_{\geq \rho}^\gamma(\supp \mu) \geq \kappa^{-1}$.
\end{lemma}

\begin{proof}
Let $\eps > 0$, and let $\{B_i\}_{i \in I}$ be a cover of $\supp \mu$ with ball $B_i$ of diameter $\delta_i \geq \rho$ and with $\sum_{i \in I} \delta_i^\gamma \leq \HC_{\geq \rho}^\gamma(\supp \mu ) + \eps$. Then
  the conclusion follows because $\eps > 0$ was arbitrary.
\end{proof}

Denote by $[X]_\delta$ the closed $\delta$-neighborhood of $X$:
\[[X]_\delta \defeq \big\{z\in[0,1] \ \big| \ \exists x\in X~\text{with}~ |z - x| \leq\epsilon\big\}.\]

\begin{lemma}
\label{lem_box_counting_estimate}
Let $a\geq 1$ and $\rho > 0$. If $X, Y \subseteq \R$ are compact and $X\subseteq [Y]_{a\rho}$, then
\[
\HC_{\geq \rho}^\gamma (X) \, \ll_a\, \HC_{\geq \rho}^\gamma (Y).
\]
\end{lemma}

\begin{proof}
Let $\{B_i\}_{i\in I}$ be a collection of open balls covering $Y$ and where $B_i$ has diameter $r_i\geq \rho$ and
$\sum_{i\in I}r_i^\gamma<2\HC_{\geq\rho}^\gamma(Y)$.
Since $X\subseteq [Y]_{a\rho}$, it follows that $X\subseteq \bigcup_{i\in I} [B_i]_{a\rho}$ and $[B_i]_{a\rho}$ is a ball of diameter $r_i+2a\rho\leq (2a+1)r_i$.
Therefore $\HC_{\geq\rho}^\gamma(X)\leq\sum_{i\in I}((2a+1)r_i)^\gamma\leq 2(2a+1)\HC_{\geq\rho}^\gamma(Y)$.
\end{proof}

\subsection{Multiplicatively invariant subsets of the reals and their finite approximations}

In this section, we record some basic facts about multiplicatively invariant subsets of $[0,1]$ and their discrete approximations.

\begin{definition} Let $r \in \N$ and $X \subseteq [0,1]$.
\begin{itemize}
    \item The map $T_r: [0,1] \to [0,1]$ is defined by $T_r x = \{rx\}$, where $\{ \cdot \}$ denotes the fractional part of a real number.
    \item The set $X$ is \emph{$\times r$-invariant} if it is closed and $T_r X \subseteq X$.
\end{itemize}
\end{definition}

The Hausdorff and Minkowski dimensions of a multiplicatively invariant set coincide.  As a consequence of this regularity, the Hausdorff dimension of products of such sets is also well-behaved.  We record these facts here for later use.

\begin{theorem}[{\cite[Proposition III.1]{furstenbergdisjointness}}]
\label{thm:furstenbergdimsofsubshift}
If $X \subseteq [0,1]$ is $\times r$-invariant, then $\dimh X = \dimm X$.
\end{theorem}

\begin{lemma}
\label{dimension_of_product_of_invariant_sets}
If $X, Y \subseteq [0,1]$ are $\times r, \times s$-invariant, respectively, then $\dimh (X \times Y) = \dimh X + \dimh Y$.
\end{lemma}

\begin{proof}
This follows immediately from \cite[Corollary 8.11]{mattila_geometry_of_sets_1995} and the fact that $\dimh X = \udimm X$.
\end{proof}

Since we will work almost exclusively with finite approximations to multiplicatively invariant sets, we establish some useful notation.

\begin{definition}\label{def_discrete_approx}
Let $X \subseteq [0,1]$ be $\times r$-invariant.  For $n \in \Nz$, the set $X_n$ denotes the set $X$ rounded down to the lattice $r^{-n}\Z$. That is, the point $i / r^n$ is an element of $X_n$ if and only if $X \cap [i/r^n, (i+1) / r^n)$ is non-empty.
\end{definition}

The next results show that finite approximations to a multiplicatively invariant set are multiplicatively invariant and are discrete models of fractal sets as captured by \cref{def_rho_gamma_sets}.

\begin{lemma}
\label{mult_invariance_of_approximations}
Let $X \subseteq [0,1]$ be $\times r$-invariant.  For all $n \in \N$, $T_r X_n \subseteq X_{n-1}$.
\end{lemma}

\begin{proof}
Let $n \in \N$, and let $i/r^n \in X_n$ with $i \in \{0,\ldots, r^n - 1\}$. Write $i = i_0 + d_{n-1} r^{n-1}$ with $i_0 \in \{0, \ldots, r^{n-1}-1\}$ and $d_{n-1} \in \{0, \ldots, r-1\}$. Note that $T_r (i / r^n) = i_0 / r^{n-1}$ and $T_r ((i+1) / r^n) = (i_0+1) / r^{n-1}$.  We must show that $i_0 / r^{n-1} \in X_{n-1}$.

Since $i/r^n \in X_n$, there exists $x \in X \cap [i/r^n, (i+1)/r^n)$. Since $T_r x \in X$, $T_r x \in X \cap [i_0/r^{n-1}, (i_0+1) / r^{n-1})$. It follows by the definition of $X_{n-1}$ that $i_0/r^{n-1} \in X_{n-1}$, as was to be shown.
\end{proof}

\begin{lemma}
\label{discrete_invariant_set_bounds}
Let $r \geq 2$, and let $X \subseteq [0,1]$ be a $\times r$-invariant set.  For all $\gamma> \dimh X$, there exists $c > 0$ such that for all sufficiently large $N \in \N$, the set $X_N$ is a $(r^{-N},\gamma)_c$-set.
\end{lemma}

\begin{proof}
Let $\gamma > \dimh X$.  Because $\gamma > \udimm X$ (cf. \cref{thm:furstenbergdimsofsubshift}), there exists $c_0 > 0$ such that for all $N \in \N$,
\begin{align}\label{mink_upper_bound}
    |X_N| \leq c_0 r^{N \gamma}.
\end{align}
Using the fact that $X$ is $\times r$-invariant, that $T_r^n$ is injective on half-open intervals of length $r^{-n}$, \cref{mult_invariance_of_approximations}, and the bound in (\ref{mink_upper_bound}), for all $0 \leq n \leq N$ and for all $i \in \{0,\ldots, r^n-1\}$,
\begin{align}
\label{eqn_rr001}
\left|X_N \cap \left[\frac i{r^n}, \frac {i+1}{r^n}\right) \right| \leq \big| T_r^n X_N \big|  \leq \big|X_{N-n} \big| \leq c_0 r^{(N-n)\gamma}.
\end{align}

Put $c = 2r^{\gamma}c_0$. To show that $X_N$ is a $(r^{-N},\gamma)_c$-set, let $B \subseteq \R$ be a ball of diameter $\delta \geq r^{-N}$. Put $n = \lfloor - \log_r \delta \rfloor$ so that $r^{-(n+1)} < \delta \leq r^{-n}$, and note that a union of two intervals of length $r^n$ of the form above suffice to cover $B$.  Therefore,
\[\big|X_N \cap B \big| \leq 2c_0r^{(N-n)\gamma} \leq c \left( \frac{\delta}{r^{-N}} \right)^{\gamma},\]
as was to be shown. 
\end{proof}

\begin{lemma}
\label{discrete_invariant_set_lowerupper_bounds}
Let $r \geq 2$, and let $X \subseteq [0,1]$ be non-empty and $\times r$-invariant.  For all $\gamma>\dimh X$ and all sufficiently large $N \in \N$,
\[
r^{N\dimh X}\leq |X_N|\leq r^{N\gamma}.
\]
\end{lemma}

\begin{proof}
Let $\gamma > \dimh X$.  Because $\gamma > \udimm X$ (cf. \cref{thm:furstenbergdimsofsubshift}), we have that $|X_N| \leq r^{N\gamma}$ for all but finitely many $N\in\N$. It remains to show the lower bound.

Let $M,N\in\N$. Since $\big[\frac i{r^{N}}, \frac {i+1}{r^{N}}\big)$, $i=0,1,\ldots,r^N-1$, forms a partition of $[0,1)$, we have 
\begin{align*}
\big|X_{N+M} \big|=\sum_{i=0}^{r^N-1}
\left|X_{N+M} \cap \left[\frac i{r^{N}}, \frac {i+1}{r^{N}}\right) \right|.
\end{align*}
Note that $X_{N+M} \cap \big[\frac i{r^{N}}, \frac {i+1}{r^{N}}\big)$ is non-empty if and only if $X \cap \big[\frac i{r^{N}}, \frac {i+1}{r^{N}}\big)$ is non-empty, which happens exactly when ${i}/{r^N}\in X_N$. Hence
\begin{align}
\label{eqn_rr002}
\big|X_{N+M} \big|=\sum_{{i}/{r^N}\in X_N}
\left|X_{N+M} \cap \left[\frac i{r^{N}}, \frac {i+1}{r^{N}}\right) \right|.
\end{align}
It follows from \eqref{eqn_rr001} that $\big|X_{N+M} \cap \big[\frac i{r^{N}}, \frac {i+1}{r^{N}}\big) \big| \leq |X_M|$, which combined with \eqref{eqn_rr002} shows that $|X_{N+M}|\leq |X_N||X_M|$. In view of this sub-additive property, it follows from Fekete's Lemma that the sequence $|X_N|^{1/N}$ converges to its infumum, i.e., 
\[
\lim_{N\to\infty}|X_N|^{1/N}=\inf_{N\in\N}|X_N|^{1/N}. 
\]
It follows from $\dimh X = \dimm X$ that $r^{\dimh X} = \lim_{N\to\infty}|X_N|^{1/N}$.  Therefore, $r^{\dimh X} = \inf_{N\in\N}|X_N|^{1/N}$, and hence $r^{N\dimh X}\leq |X_N|$ for all $N\in\N$, as desired.
\end{proof}

The following notation, borrowed from \cite{peresshmerkin2009}, allows us to easily compare powers of $r$ and powers of $s$.  This is useful when considering the finite approximations to the Cartesian product of a $\times r$- and a $\times s$-invariant set.

\begin{definition}
\label{def_n_prime}
For $n \in \Nz$, we set $n'=\lfloor n\log r/\log s\rfloor$ to be the greatest integer so that $s^{n'} \leq r^{n}$.  (The bases $r$ and $s$ do not appear in this notation but should always be clear from context.)
\end{definition}

Recall from \cref{def_discrete_approx} that $X_N$ is the set $X$ rounded to the lattice $r^{-N}\Z$.  Extending this notation to $Y$, the set $Y_{N}$ is the set $Y$ rounded to the lattice $s^{-N} \Z$.  Since $r^{-N}$ is approximately equal to $s^{-N'}$ (where $N'$ is as defined in \cref{def_n_prime}), the set $Y_{N'}$ is the discrete approximation to $Y$ that is on a scale closest to the scale of $X_N$.  Therefore, the sets $X_N$ and $Y_{N'}$ will always be considered in the same context, as opposed to the sets $X_N$ and $Y_N$.

\begin{corollary}
\label{product_of_discrete_invariant_sets_bounds}
Let $2 \leq r < s$, let $X, Y \subseteq [0,1]$ be non-empty $\times r$- and $\times s$-invariant sets. For all $\xi> \dimh X+\dimh Y$, there exist $c_1, c_2 > 0$ and $M_0\in\N$ such that for all $N \geq M_0$, the sets $X_N \times Y_{N'}$ and $X_N \times Y_{N'+1}$ are $(c_1r^{-N},\xi)_{c_2}$-sets satisfying $r^{N (\dimh X+\dimh Y)} \leq |X_N \times Y_{N'}| \leq r^{N \xi}$ and $r^{N (\dimh X+\dimh Y)} \leq |X_N \times Y_{N'+1}| \leq r^{N \xi}$.
\end{corollary}

\begin{proof}
Let $\xi> \dimh X+\dimh Y$. Let $g > \dimh X $ and $h > \dimh Y $ be such that
\[\dimh (X \times Y) < g + h < \xi.\]
Applying \cref{discrete_invariant_set_bounds} and \cref{discrete_invariant_set_lowerupper_bounds}, there exist $c, d > 0$ such that for sufficiently large $N \in \N$, the set $X_N$ is a $(r^{-N},g)_c$-set satisfying $r^{N \dimh X} \leq |X_N| \leq r^{N g}$ and $Y_{N'}$ is a $(s^{-N'},h)_{d}$-set satisfying $s^{N' \dimh Y} \leq |Y_{N'}| \leq s^{N' h}$.
Since $r^{N (\dimh X+\dimh Y)}= r^{N\dimh X}r^{N\dimh Y}\geq r^{N\dimh X}s^{N'\dimh Y}$, $|X_N\times Y_N|\leq |X_N\times Y_{N+1}|$ and $g+h<\xi$,
it follows that for sufficiently large $N \in \N$, one has $r^{N (\dimh X+\dimh Y)} \leq |X_N \times Y_{N'}| \leq r^{N \xi}$ and $r^{N (\dimh X+\dimh Y)} \leq |X_N \times Y_{N'+1}| \leq r^{N \xi}$.

Set $c_1 = s^{-1}$ and $c_2 = s^{g} c d$. Since $s^{N'} < r^N < s^{N'+1}$, the sets $X_N \times Y_{N'}$ and $X_N \times Y_{N'+1}$ are $c_1 r^{-N}$-separated.  Since $X_N$ is a $(r^{-N},g)_c$-set, it is a $(c_1 r^{-N},g)_{s^{g} c}$-set.\footnote{More generally, if $0 < c_1 < 1$, then every $(\delta, \gamma)_c$-set is a $(c_1 \delta, \gamma)_{c_1^{-\gamma}c}$-set.  This is a quick exercise left to the reader.} Let $B \subseteq \R^2$ be a ball of diameter $\delta \geq c_1 r^{-N}$. Note that
\begin{align*}
    \big|(X_N \times Y_{N'}) \cap B \big| &\leq s^{g} c \left( \frac{\delta}{c_1 r^{-N}} \right)^{g} d \left( \frac{\delta}{s^{-N'}} \right)^{h} \\
    &\leq s^{g} c \left( \frac{\delta}{c_1 r^{-N}} \right)^{g} d c_1^{h} \left( \frac{\delta}{c_1 r^{-N}} \right)^{h} \leq c_2 \left( \frac{\delta }{c_1 r^{-N}} \right)^{\xi},
\end{align*}
which shows that the set $X_N \times Y_{N'}$ is a $(c_1r^{-N},\xi)_{c_2}$-set. By a similar calculation,
\begin{align*}
    \big|(X_N \times Y_{N'+1}) \cap B \big| &\leq s^{g} c \left( \frac{\delta}{c_1 r^{-N}} \right)^{g} d \left( \frac{\delta}{s^{-(N'+1)}} \right)^{h} \\
    &\leq s^{g} c \left( \frac{\delta}{c_1 r^{-N}} \right)^{g} d \left( \frac{\delta}{c_1 r^{-N}} \right)^{h} \leq c_2 \left( \frac{\delta }{c_1 r^{-N}} \right)^{\xi},
\end{align*}
which shows that the set $X_N \times Y_{N'+1}$ is a $(c_1r^{-N},\xi)_{c_2}$-set.
\end{proof}

\subsection{A quantitative equidistribution lemma} \label{section_quant_equid}

The main result in this short section, \cref{cor:higherunifdistupdated}, gives a lower bound on the number of visits of an equidistributed sequence to a set as a function only of the measure and topological complexity of the set's complement.  This result is certainly not new; we state it explicitly here for convenience in a way that highlights the uniformity in the quantifiers.

For $U \in \N$, denote by $\mathcal{I}_U$ the collection of those subsets of $[0,1)$ that are a union of no more than $U$ disjoint intervals of the form $[a,b)$.

\begin{lemma}\label{lem:higherunifdistupdated}
For any uniformly distributed sequence $(x_n)_{n \in \Nz} \subseteq [0,1)$, $U \in \N$, and $\eps > 0$, there exists $N_0 \in \N$ such that for all $N \geq N_0$ and all $B \in \mathcal{I}_U$,
\[\frac 1N \big| \{ 0 \leq n \leq N-1 \ | \ x_n \in B \} \big | \leq \Leb(B) + \eps.\]
\end{lemma}

\begin{proof}
Let $(x_n)_{n \in \Nz} \subseteq [0,1)$ be uniformly distributed, $U \in \N$, and $\eps > 0$. The discrepancy of $(x_n)_{n=0}^{N-1}$ (cf. \cite[Ch. 2, Def. 1.1]{Kuipers_Niederreiter_UDbook}) is
\[D_N = \sup_{I} \left| \frac{\{0 \leq n \leq N-1 \ | \  x_n \in I \}}{N} - \Leb(I)\right|,\]
where the supremum is taken over all half-open intervals $I$ in $[0,1)$. Because $(x_n)_n$ is uniformly distributed, $D_N \to 0$ as $N \to \infty$ (cf. \cite[Ch. 2, Thm. 1]{Kuipers_Niederreiter_UDbook}).
By the definition of discrepancy, for any half-open interval $I \subseteq [0,1)$,
\[ \frac 1N \big| \{ 0 \leq n \leq N-1 \ | \ x_n \in I \} \big | \leq \Leb(I) + D_N.\]
It follows that for every $B \in \mathcal{I}_U$,
\[ \frac 1N \big| \{ 0 \leq n \leq N-1 \ | \ x_n \in B \} \big | \leq  \Leb(B) + U D_N.\]
Let $N_0 \in \N$ be large enough so that for all $N \geq N_0$, $U D_N \leq \eps$.  The conclusion follows.
\end{proof}

\begin{lemma}\label{cor:higherunifdistupdated}
Let $\beta > 0$. For any uniformly distributed sequence $(x_n)_{n \in \Nz} \subseteq [0,\beta)$ with respect to the Lebesgue measure, $U \in \N$, and $\eps > 0$, there exists $N_0 \in \N$ such that for all $N \geq N_0$ and all $J \subseteq [0,\beta)$ whose complement is covered by a union of no more than $U$ many disjoint, half-open intervals of total Lebesgue measure less than $\eps \beta / 2$,
\[\frac 1N \big| \{ 0 \leq n \leq N-1 \ | \ x_n \in J \} \big | \geq 1-\eps.\]
\end{lemma}

\begin{proof}
Let $(x_n)_{n \in \Nz} \subseteq [0,\beta)$ be uniformly distributed, $U \in \N$, and $\eps > 0$. Let $N_0$ be from \cref{lem:higherunifdistupdated} with $(x_n / \beta)_{n \in \Nz}$, $U$, and $\eps / 2$.

Let $N \geq N_0$ and $J \subseteq [0,\beta)$. Put $B = [0,\beta) \setminus J$, and note that by assumption, $B / \beta \in \mathcal{I}_U$ and $\Leb(B / \beta) < \eps / 2$. It follows from \cref{lem:higherunifdistupdated} that
\[\frac 1N \big| \{ 0 \leq n \leq N-1 \ | \ x_n / \beta \in B / \beta \} \big | < \eps.\]
Therefore,
\[\frac 1N \big| \{ 0 \leq n \leq N-1 \ | \ x_n \in J \} \big | \geq 1- \eps,\]
as was to be shown.
\end{proof}

\section{A discrete Marstrand projection theorem}
\label{section_projections_and_discrete_marstrand}

In this section, we prove a discrete analogue of Marstrand's projection theorem from geometric measure theory.  The theorem -- stated for sumsets in the introduction as \cref{theorem_marstrand} -- says that for every Borel set $A \subseteq [0,1]^2$, for Lebesgue-a.e. $\theta \in [0,\pi)$, $\dimh \pi_\theta A = \min(1,\dimh A)$, where $\pi_\theta: \R^2 \to \R^2$ is the orthogonal projection onto $\ell_\theta$, the line that contains the origin and forms an angle $\theta$ with the positive $x$-axis. Marstrand's theorem and its relatives have enjoyed much recent attention: we refer the interested reader to the survey \cite{falconer_projection_survey} and to the end of this section where we put \cref{thm:discretemarstrandbase} into more context.

The key idea behind Marstrand's theorem is that of ``geometric transversality'' and is captured in the following lemma.  The proof follows from a simple geometric argument and is left to the reader.  An immediate consequence of the lemma is that there are not many projections which map two distant points close together.

\begin{lemma}\label{lem:transversality}
For all nonzero $x \in \R^2$ and all $\rho > 0$, the set of angles $\theta \in [0,\pi)$ for which $|\pi_\theta x| \leq \rho$ is contained in at most two balls of diameter $\ll \rho |x|^{-1}$.
\end{lemma}

The results in this section add to a number of other discrete Marstrand-type theorems in the recent literature: \cite[Lemma 5.2]{limamoreira}, \cite[Prop.~3.2]{lima_moreira_comb_proof_of_marstrand}, \cite[Lemma 3.8]{glasscock_marstrand_for_integers}, \cite[Prop.~7]{peresshmerkin2009}, \cite[Prop.~4.10]{Orponen_2015} to name a few.  Let us highlight some distinguishing features of \cref{lem:transversality} and  \cref{thm:discretemarstrandbase} that play an important role in this work.  Analogues of \cref{lem:transversality} more commonly found in the literature, such as the one in \cite[Lemma 3.11]{mattila_geometry_of_sets_1995}, bound the \emph{measure} of the set of projections which map $x$ close to 0. The result in \cref{lem:transversality} uses coverings to capture topological information on the set of projections.  This information is carried into \cref{thm:discretemarstrandbase} and is important in the application to \cref{main_real_theorem_intro}. Another useful feature of \cref{thm:discretemarstrandbase} is the allowance of a subset $A'$ in \eqref{eqn:defofexceptionalset}; this will allow us to treat sets in \cref{main_real_theorem_intro} that exhibit multiplicative invariance without necessarily being self-similar.

\subsection{A discrete projection theorem}

Our discrete analogue of Marstrand's theorem, \cref{thm:discretemarstrandbase}, reaches a conclusion similar to that of Marstrand's by quantifying the size of the set $E$ of exceptional directions, those directions in which the image of the set $A$ is small.  On a first reading, it is safe to think of $\gamma < 1$, $n \approx \rho^{-\gamma}$, $\delta = 1$, and $m \approx \rho^{-(\gamma - \eps)}$.  In this case, the set $A$ is a discrete analogue of a set of Hausdorff dimension $\gamma$ and the set $E$ is the set of exceptional directions in which the set $A$ loses at least a proportion $\rho^{\eps}$ of its points.

\begin{theorem}
\label{thm:discretemarstrandbase}
Let $\gamma, \rho, c > 0$. Put $\overline{\gamma} = \min(\gamma,1)$. If $A \subseteq [0,1]^2$ is a $(\rho,\gamma)_c$-set with $n \defeq |A| > -\log c$, then for all $\delta > 0$ and all $0 \leq m \leq \delta^2 n \big/ 4$, the set
\begin{align}\label{eqn:defofexceptionalset}E = \big\{ \theta \in [0,\pi) \ \big | \ \exists A' \subseteq A, \ |A'| \geq \delta n, \ \NC (\pi_\theta {A'},\rho) \leq m \big \}\end{align}
satisfies
\[\NC ( E, \rho) \ll_{\gamma,c} \rho^{-1} \frac{m}{\delta^2 n} \begin{dcases} n^{1- \overline{\gamma}/\gamma} & \text{ if } \gamma \neq 1 \\ \log n & \text{ if } \gamma = 1 \end{dcases}.\]
\end{theorem}

\begin{proof}
Let $A \subseteq [0,1]^2$ be a $(\rho,\gamma)_c$-set of cardinality $n > - \log c$. Let $\delta > 0$, and let $0 \leq m \leq \delta^2 n \big/ 4$.

Define $S(\theta) = \big\{ (a_1,a_2) \in A^2 \ \big| \ |\pi_\theta (a_1-a_2) | < \rho \big\}$. Let $E'$ be a maximal $\rho$-separated subset of $E$; thus, $|E'| = \NC (E,\rho)$. The goal is to bound $\sum_{\theta \in E'} \big| S(\theta) \big|$ from above and below to get the desired bound on $|E'|$.

Let $\theta \in E'$ and $A'$ be the subset of $A$ corresponding to $\theta$. Since the set $\pi_{\theta}{A'}$ lies on a line and $\NC (\pi_{\theta}{A'},\rho) \leq m$, there exists a collection $\{B\}_{B \in \mathcal{B}}$ of no more than $2m$ closed balls $B$ of diameter $\rho$ whose union covers $\pi_{\theta}{A'}$. By Cauchy-Schwarz,
\begin{align*}
(\delta n)^2  \leq |A'|^2 &\leq \left( \sum_{B \in \mathcal{B}}  \big| \{ a_0 \in A' \ | \ \pi_\theta {a_0} \in B \} \big| \right)^2 \\
&\leq \big| \mathcal{B} \big| \sum_{B \in \mathcal{B}} \big| \{ a_0 \in A' \ | \ \pi_\theta {a_0} \in B \} \big|^2 \\
&\leq 2m \sum_{B \in \mathcal{B}} \big| \{ a \in A \ | \ \pi_\theta a \in B \} \big|^2\\
&= 2m \sum_{B \in \mathcal{B}} \big| \{ (a_1,a_2) \in A^2 \ | \ \pi_\theta {a_1}, \pi_\theta {a_2} \in B \} \big|\\
&\leq 2m \big| S(\theta) \big|.
\end{align*}
It follows that
\begin{align}\label{eqn:lowerboundonenergysum} \frac{\delta^2n^2}{2m}|E'| \leq \sum_{\theta \in E'} \big| S(\theta) \big|. \end{align}

Now we use Lemma \ref{lem:transversality} to bound the right hand side of (\ref{eqn:lowerboundonenergysum}) from above: for $a_1,a_2 \in [0,1]^2$, the set
\[\Theta (a_1,a_2) = \big\{ \theta \in [0,\pi) \ \big| \ \big| \pi_{\theta}{(a_1-a_2)} \big| < \rho \big\}\]
is contained in at most two balls of diameter $\ll \rho / |a_1 - a_2|$. Therefore, $\NC (\Theta(a_1,a_2),\rho) \ll 1 / |a_1 - a_2| $, and using the fact that $E'$ is $\rho$-separated, we see that

\[\sum_{\theta \in E'}1_{S(\theta)}(a_1,a_2) = \sum_{\theta \in E'}1_{\Theta(a_1,a_2)}(\theta)\leq K \frac{1}{|a_1 - a_2|}\]
for some constant $K$ depending on the result in Lemma \ref{lem:transversality}. It follows that
\begin{align*}
\sum_{\theta \in E'} \big| S(\theta) \big| &= \sum_{\theta \in E'} \sum_{a_1,a_2 \in A} 1_{S(\theta)}(a_1,a_2)\\
& = n|E'| + \sum_{\substack{a_1, a_2 \in A \\ a_1 \neq a_2}} \sum_{\theta \in E'}1_{\Theta(a_1,a_2)}(\theta)\\
& \leq n|E'| + K \sum_{\substack{a_1, a_2 \in A \\ a_1 \neq a_2}} |a_1-a_2|^{-1},
\end{align*}
and so we are left to bound the second term from above.

For $\ell \in \Nz$, let $H_\ell = \{x \in \R^2 \ | \ |x| \in [\rho e^\ell, \rho e^{\ell+1}) \}$.  Breaking up the sum $\sum |a_1-a_2|^{-1}$ by fixing $a_1$ and partitioning the $a_2$'s by shells, and using the fact that $A$ is $\rho$-separated, we see
\begin{align*}
\sum_{\substack{a_1, a_2 \in A \\ a_1 \neq a_2}} |a_1-a_2|^{-1} 
&=
\sum_{a_1 \in A} \sum_{\ell = 0}^{\infty} \sum_{a_2 \in A \cap (a_1 + H_\ell)} |a_1-a_2|^{-1}\\
&\leq
\rho^{-1} \sum_{a_1 \in A} \sum_{\ell = 0}^{\infty} e^{-\ell } \big|A \cap (a_1 + H_\ell)\big|.
\end{align*}
Since $A$ is a $(\rho,\gamma)_c$-set, for all $\ell \geq 0$, $\big|A \cap (a_1 + H_\ell)\big| \leq c \left(2\rho e^{\ell+1} \big/ \rho \right)^\gamma$.
On the other hand, $\sum_{\ell=0}^\infty \big|A \cap (a_1 + H_\ell)\big|=|A|-1$.
It follows then from the fact that $\ell\mapsto e^{-\ell}$ is decreasing that $\sum_{\ell = 0}^{\infty} e^{-\ell } \big|A \cap (a_1 + H_\ell)\big|\leq \sum_{\ell = 0}^{\ell_0}2^\gamma ce^{\ell(\gamma-1)+\gamma}$, where $\ell_0 = \lceil \log ((n/c)^{1/\gamma}) \rceil$ is the smallest value such that the set $A$ could be contained in a ball of diameter $\rho e^{\ell_0}$ about $a_1$. Therefore,

\begin{align*}
\rho^{-1} \sum_{a_1 \in A} \sum_{\ell = 0}^{\infty} e^{-\ell } \big|A \cap (a_1 + H_\ell)\big| &\ll_{\gamma,c}
\rho^{-1} \sum_{a_1 \in A} \sum_{\ell = 0}^{\ell_0} \big(e^{\gamma - 1}\big)^\ell
\\
&\ll_{\gamma,c} 
\rho^{-1} n \begin{dcases} n^{1- \overline{\gamma}/\gamma} & \text{ if } \gamma \neq 1 \\ \log n & \text{ if } \gamma = 1 \end{dcases}.
\end{align*}

Combining the upper and lower bounds on $\sum_{\theta \in E'} \big| S(\theta) \big|$, we see that there exists a constant $K$ depending on the result in Lemma \ref{lem:transversality}, $\gamma$, and $c$ such that
\[\frac{\delta^2n^2}{2m}|E'| \leq n|E'| + K \rho^{-1}n \begin{dcases} n^{1- \overline{\gamma}/\gamma} & \text{ if } \gamma \neq 1 \\ \log n & \text{ if } \gamma = 1 \end{dcases}.\]
Dividing both sides by $n$ and using the fact that $m \leq \delta^2n / 4$, we see that
\[ \frac{\delta^2n}{4m}|E'| \leq \left( \frac{\delta^2n}{2m} - 1 \right) |E'| \leq K \rho^{-1} \begin{dcases} n^{1- \overline{\gamma}/\gamma} & \text{ if } \gamma \neq 1 \\ \log n & \text{ if } \gamma = 1 \end{dcases},\]
which rearranges to the desired conclusion.
\end{proof}

\subsection{A corollary for oblique projections}

The proof of \cref{main_real_theorem_intro} will feature oblique projections instead of orthogonal ones. The following corollary concerns oblique projections and is stated in a way that will make it immediately applicable in the proof of \cref{main_real_theorem_intro}.

Denote by $\Pi_{t}: \R^2 \to \R$ the oblique projection $\Pi_{t}(x,y) = x + t y$.  Let $\varphi: (0, \pi / 2) \to \R$ be the diffeomorphism $\varphi(\theta) = \log \tan \theta$.  Note that $\Pi_{e^{\varphi(\theta)}}$ is the oblique projection that is the ``continuation'' of the orthogonal projection $\pi_\theta$, meaning that the points $(x,y)$, $(\Pi_{e^{\varphi(\theta)}}(x,y),0)$, and $\pi_\theta (x,y)$ are collinear.

\begin{corollary}\label{thm:discretemarstrandcorollary}
Let $0 < \gtwo < \gthree < \gfour < \gfive$ be such that $\gtwo < 1$ and
\begin{align}\label{inequality_between_gammas_for_marstrand}
    2 (\gfive - \gthree) < \gfour - \gtwo.
\end{align}
For all compact $I \subseteq \R$, all $\eps, c_1, c_2, c_3 > 0$, all sufficiently small $\rho > 0$ (depending on all previous quantities), and all $(c_1 \rho,\gfive)_{c_2}$-sets $A \subseteq [0,1]^2$ with $|A| \geq \rho^{-\gfour}$, there exists $T \subseteq I$ with the following properties:
\begin{enumerate}[label=(\Roman*)]

\item \label{marstrand_corollary_conclusion_one} the set $I \setminus T$ can be covered by a disjoint union of not more than $\eps \rho^{-1} / 2$-many half-open intervals of length $\rho$, a cover of total Lebesgue measure less than $\eps$.
\item \label{marstrand_corollary_conclusion_three} for all $t \in T$ and all $A' \subseteq A$ with $|A'| \geq \rho^{-\gthree}$, there exists a subset $A'_t \subseteq A'$ with $|A'_t| \geq \rho^{-\gtwo}$ such that the points of $\Pi_{e^t} A'_t$ are distinct and $c_3 \rho$-separated.
\end{enumerate}
\end{corollary}

\begin{proof}
Let $I \subseteq \R$ be compact and $\eps, c_1, c_2, c_3 > 0$. Let $\sigma \in \big(\gfive - \gthree, (\gfour-\gtwo)/2 \big)$. Let $\rho > 0$ be sufficiently small (to be specified later, but depending only on the quantities introduced thus far).  Let $A \subseteq [0,1]^2$ be a $(c_1 \rho,\gfive)_{c_2}$-set with $|A| \geq \rho^{-\gfour}$. Put $\overline{\gfive} = \min(1,\gfive)$, $n = |A|$, $\delta = \rho^\sigma$, and $m = 2c_3 \rho^{-\gtwo}$.  Note that since $A$ is a $(c_1 \rho,\gfive)_{c_2}$-set contained in a ball of diameter $\sqrt{2}$, $n \leq 2 c_2 (c_1 \rho)^{-\gfive}$.

We want to apply \cref{thm:discretemarstrandbase} with $\gfive$ as $\gamma$, $c_1 \rho$ as $\rho$, $c_2$ as $c$, and with $A$, $n$, $\delta$, and $m$ as they are. We see that the inequality $n > - \log c_2$ holds for $\rho$ sufficiently small, as does $m \leq \delta^2 n / 4$ since $\sigma < (\gfour-\gtwo)/2$. Since the conditions of \cref{thm:discretemarstrandbase} hold, the set $E \subseteq [0,\pi)$ defined in (\ref{eqn:defofexceptionalset}) satisfies
\begin{align}\label{corollary_to_marstrand_bound_on_E}
    \begin{aligned}\NC(E,\rho) &\ll_{\gfive, c_2} \rho^{-1} \frac{m}{\delta^2 n} n^{1-\overline{\gfive} / \gfive} \log n \\
    &\ll_{\gfive, c_1, c_2,c_3} \rho^{-1} \frac{\rho^{-\gtwo}}{\rho^{2\sigma} \rho^{-\gfour \overline{\gfive} / \gfive}} \log \left( \rho^{-\gfive} \right).
    \end{aligned}
\end{align}

Let $J = \varphi^{-1}(I)$, and put $T = I \setminus \restr{\varphi}{J}(E)$. Since the map $\restr{\varphi}{J}$ is bi-Lipschitz,
\[\NC(\restr{\varphi}{J}(E), \rho) \asymp_I \NC(E,\rho).\]
Combining this with \eqref{corollary_to_marstrand_bound_on_E} and the fact that $\sigma < (\gfour-\gtwo)/2$, we have that for sufficiently small $\rho$, $\NC(I \setminus T,\rho) \leq \eps \rho^{-1} / 6$. It follows that the set $I \setminus T$ can be covered by a disjoint union of not more than $\eps \rho^{-1} / 2$-many half-open intervals of length $\rho$, a cover of total measure less than $\eps$. This establishes \ref{marstrand_corollary_conclusion_one}.

To prove \ref{marstrand_corollary_conclusion_three}, let $t \in T$, and let $A' \subseteq A$ with $|A'| \geq \rho^{-\gthree}$. Since $n \leq 2 c_2 (c_1 \rho)^{-\gfive}$ and $\sigma > \gfive - \gthree$, for sufficiently small $\rho$, $\rho^{-\gthree} \geq \delta n$.  It follows that $|A'| \geq \delta n$. 
Because $\theta \defeq \varphi^{-1}(t) \not\in E$, $\NC (\pi_\theta {A'},\rho) \geq m$. It follows that $\NC (\pi_\theta {A'},c_3\rho) \geq \rho^{-\gtwo}$. By choosing points in $A'$ in each fiber of a maximally $\rho$-separated set of the projection, we see that there exists a subset $A'_t \subseteq A'$ of cardinality at least $\rho^{-\gtwo}$ such that the orthogonal projection of the points in $A'_\theta$ onto $\ell_\theta$ are disjoint and $c_3\rho$-separated. Since the oblique projection $\Pi_{e^t}$ increases distances between points that lie on $\ell_\theta$, the images of points of $A'_t$ under $\Pi_{e^t}$ are $c_3\rho$-separated.
\end{proof}

\section{Trees and a subtree regularity theorem}
\label{tree_section}

Trees are combinatorial objects that are convenient for describing fractal sets.  We will be concerned solely with finite trees throughout this work.  After giving the main definitions, we motivate their importance by explaining how they will be used in the proof of \cref{main_real_theorem_intro}.  We move then to prove the main result in this section.

\subsection{Preliminary definitions}

The following definitions describe the familiar notion of a rooted tree, a graph with no cycles whose vertices can be arranged on levels and whose edges only connect vertices on adjacent levels.

\begin{definition}\label{def_tree} \leavevmode
\begin{itemize}
    \item A \emph{tree of height $N \in \Nz$} is a finite set of \emph{nodes} $\Gamma$ together with a partition $\Gamma=\Gamma_0\cup\cdots\cup\Gamma_N$ with $|\Gamma_0|=1$ and a \emph{parent} function $P: \Gamma \setminus \Gamma_0 \to \Gamma \setminus \Gamma_N$ such that for every $n\in\{1,\dots,N\}$, $P(\Gamma_n)=\Gamma_{n-1}$.
    \item The nodes in $\Gamma_n$ have \emph{height} $n$. The single node with height $0$ is the \emph{root} and the nodes with height $N$ are called \emph{leaves}.
    \item The node $Q$ is the \emph{parent} of each of its \emph{children}, nodes in the set $C_\Gamma(Q) \defeq P^{-1}(Q)$.
    \item If $Q$ is a node of height $n$, the \emph{induced tree based at $Q$} is the tree $\Gamma_Q \defeq \cup_{i=0}^{N-n} C_\Gamma^{i}(Q)$ of height $N-n$ with root $Q$ and the same parent function as $\Gamma$, restricted to the set $\Gamma_Q$.
    \item A \emph{subtree} of $\Gamma$ is a tree $\Gamma' \subseteq \Gamma$ of the same height as $\Gamma$ with parent function $\restr{P}{\Gamma' \setminus \Gamma_0'}$. (A subtree is uniquely determined by its non-empty set of leaves $\Gamma_N' \subseteq \Gamma_N$.)
\end{itemize}
\end{definition}

Continuing with terminology inspired by genealogy trees, the ancestors of a node $Q$ are those nodes that lie between $Q$ and the root. For the reasons described below in \cref{remark_reason_for_fertile_ancestry}, it will be important to count the number of ancestors of $Q$ that have many children.  To this end, we introduce the following terminology and notation.

\begin{definition} Let $\Gamma$ be a tree, $c > 0$, and $\omega \in [0,1]$.
\begin{itemize}
    \item The \emph{ancestry} of $Q \in \Gamma_n$ is the set
    \[\AC_\Gamma(Q) \defeq \{P^k(Q) \ | \ 1 \leq k \leq n\}.\]
    Note that $|\AC_\Gamma(Q)|$ is equal to the height of $Q$.
    \item The node $Q$ is \emph{$c$-fertile} if $|C_\Gamma(Q)| \geq c$. The set of $c$-fertile ancestors of $Q$ is denoted
    \[\FC_{\Gamma,c}(Q) \defeq \{A \in \AC_\Gamma(Q) \ | \ \text{$A$ is $c$-fertile} \}.\]
    A node $Q$ has \emph{$(c,\omega)$-fertile ancestry} if $|\FC_{\Gamma,c}(Q)| \geq \omega |\AC_\Gamma(Q)|$.
\end{itemize}
\end{definition}

The following definitions allow us to capture the dimension of a finite tree by giving costs to the nodes and measuring the cost of the least expensive cut.

\begin{definition} Let $\Gamma$ be a tree, $r \in \N$, $r \geq 2$, and $\gamma > 0$.
\begin{itemize}
    \item A \emph{cut} of $\Gamma$ is a subset $\CC \subseteq \Gamma$ such that for every leaf $L$ of $\Gamma$, $\big(\{L\} \cup \AC_\Gamma(L) \big) \cap \CC \neq \emptyset$.
    \item The \emph{$\gamma$-Hausdorff content of $\Gamma$ with base $r$} is
    \[\HC_r^\gamma(\Gamma) \defeq \min \left\{ \sum_{Q \in \mathcal{C}} r^{- \text{height}(Q) \gamma} \ \middle | \ \text{$\CC$ is a cut of $\Gamma$} \right\}.\]
\end{itemize}
\end{definition}

The main result in this section, \cref{tree_thinning_theorem}, says, roughly speaking, that any tall enough tree with Hausdorff content bounded from below and with a uniform upper bound on the number of children of any node has a subtree in which most nodes have fertile ancestry.  Before making this statement precise and beginning with the details of the proof, let us make two observations about the concept of fertile ancestry that will help explain why it will be useful later on in the proof of \cref{main_real_theorem_intro}.

\begin{remark} \leavevmode \label{remark_reason_for_fertile_ancestry}
\begin{enumerate}[label=(\Roman*)]
\item \label{reason_for_fertile_ancestry_one} The property of having fertile ancestry is preserved under a type of tree thinning process that we will employ in the proof of \cref{main_real_theorem_intro}.  More specifically, suppose that $\Gamma$ is a tree in which every node has either one child or at least $c$ many children and in which every node has $(c, \omega)$-fertile ancestry. Suppose further that for every node $Q$, there exists a subset $\tilde{C}(Q) \subseteq C_\Gamma (Q)$ of the children of $Q$ with $|\tilde{C}(Q)| \geq \min \big(\tilde c,|C_\Gamma (Q)| \big)$.  These subsets naturally give rise to a subtree $\tilde \Gamma$ obtained by thinning the tree $\Gamma$: the subtree $\tilde \Gamma$ is uniquely defined by the property that if $Q$ is a node of $\tilde \Gamma$, then $C_{\tilde \Gamma}(Q) = \tilde{C}(Q)$. It is not hard to see that every node in $\tilde \Gamma$ has $(\tilde c, \omega)$-fertile ancestry, regardless of how the subsets of children $\tilde{C}(Q)$ were chosen.
\item \label{reason_for_fertile_ancestry_two} A tree in which every node has fertile ancestry necessarily has large Hausdorff content.  This is a simple consequence of the mass distribution principle (or the max flow-min cut theorem) for trees, the real analogue of which is stated in \cref{lem_consequence_of_discrete_frostman}. More specifically, let $\Gamma$ be a tree, and consider a ``flow'' through $\Gamma$ of magnitude 1 starting at the root that splits equally amongst children.  The value of the flow at any node $Q$ with fertile ancestry can be bounded from above using the fact that many times, much of the flow is split amongst a large set of children before reaching $Q$. If all nodes of $\Gamma$ have fertile ancestry, then the flow is not concentrated too highly at any node.  According to the mass distribution principle, the Hausdorff content of a tree that supports such a flow is high.
\end{enumerate}
\end{remark}

\subsection{A subtree regularity theorem}

We now proceed with the main results in this subsection.  In the next two results, fix $r \geq 2$ and $0 < \gthree < \gfour < \gfive$ such that setting
\begin{align*}
    A &\defeq \gfive - \gfour + \log_r 2,\\
    B &\defeq \gfour - \gthree - \log_r 2,
\end{align*}
ensures the quantity $B$ is positive. 
The following lemma describes the fundamental dichotomy behind \cref{tree_thinning_theorem}.

\begin{lemma}\label{lemma_tree_dichotomy_for_capacity}
If $\Gamma$ is a tree with the property that
\begin{align}\label{soft_one_child_policy}
    \text{every node in the tree has at most $r^{\gfive}$ many children,}
\end{align}
then at least one of the following holds:
\begin{enumerate}[label=(\Roman*)]
    \item \label{tree_dichotomy_case_I} there are at least $r^{\gthree}$ many children $Q$ of the root, each of which satisfies
    \[\HC_r^{\gfour}(\Gamma_Q) \geq \HC_r^{\gfour}(\Gamma) r^{-A};\]
    \item \label{tree_dichotomy_case_II} there is at least one child $Q$ of the root satisfying
    \[\HC_r^{\gfour}(\Gamma_Q) \geq \HC_r^{\gfour}(\Gamma) r^{B}.\]
\end{enumerate}
\end{lemma}

\begin{proof}
Let $\Gamma$ be a tree satisfying (\ref{soft_one_child_policy}). Let $Q_1$, $Q_2$, \dots, $Q_I$ be the children of the root of $\Gamma$, ordered so that $\HC_r^{\gfour}(\Gamma_{Q_i}) \geq \HC_r^{\gfour}(\Gamma_{Q_{i+1}})$. If neither \ref{tree_dichotomy_case_I} nor \ref{tree_dichotomy_case_II} holds, then $\HC_r^{\gfour}(\Gamma_{Q_1}) < \HC_r^{\gfour}(\Gamma) r^{B}$ and $\HC_r^{\gfour}(\Gamma_{Q_{\lceil r^{\gthree} \rceil}}) < \HC_r^{\gfour}(\Gamma) r^{-A}$.  It follows by the ordering of the $Q_i$'s and the definition of the Hausdorff content and induced trees that
\begin{align*}
    \HC_r^{\gfour}(\Gamma) &\leq r^{-\gfour} \sum_{i=1}^I \HC_r^{\gfour}(\Gamma_{Q_i}) \\
    &= \sum_{i=1}^{\lfloor r^{\gthree} \rfloor} r^{-\gfour} \HC_r^{\gfour}(\Gamma_{Q_i}) + \sum_{i=\lceil r^{\gthree} \rceil}^{I} r^{-\gfour} \HC_r^{\gfour}(\Gamma_{Q_i})\\
    &< r^{\gthree} r^{-\gfour} \HC_r^{\gfour}(\Gamma) r^{B} + r^{\gfive} r^{-\gfour} \HC_r^{\gfour}(\Gamma) r^{-A} = \HC_r^{\gfour}(\Gamma),
\end{align*}
a contradiction.
\end{proof}

\begin{lemma}\label{pre_tree_thinning_theorem}
Every finite tree $\Gamma$ that satisfies (\ref{soft_one_child_policy}) has a subtree $\Gamma'$ with the property that for all nodes $Q$ in $\Gamma'$,
\begin{align}\label{fertile_ancestry_with_ell_help}
    \big|\FC_{\Gamma',r^{\gthree}}(Q) \big| \geq \frac{|\AC_{\Gamma'}(Q)|B + \log_r \HC_r^{\gfour}(\Gamma)}{A+B}.
\end{align}
\end{lemma}

\begin{proof}
We will prove the lemma by induction on the height $N$ of the tree $\Gamma$.  To verify the base case, let $\Gamma$ be the tree of height $N=0$: a single node with no children.  Taking $\Gamma' = \Gamma$, the inequality (\ref{fertile_ancestry_with_ell_help}) for this single node follows from the fact that $\log_r \HC_r^{\gfour}(\Gamma) = 0$.

Suppose that $N \in \N$ is such that the theorem holds for all trees of height $N-1$. Let $\Gamma$ be a tree of height $N$ that satisfies (\ref{soft_one_child_policy}). By \cref{lemma_tree_dichotomy_for_capacity}, at least one of Case \ref{tree_dichotomy_case_I} or Case \ref{tree_dichotomy_case_II} holds.

Suppose Case \ref{tree_dichotomy_case_I} of \cref{lemma_tree_dichotomy_for_capacity} holds. Let $Q$ be any one of the $r^{\gthree}$-many children guaranteed by Case \ref{tree_dichotomy_case_I}. By the induction hypothesis, there exists a subtree $\Gamma_Q'$ of $\Gamma_Q$ in which every node satisfies (\ref{fertile_ancestry_with_ell_help}) with $\Gamma_Q$ in place of $\Gamma$ and $\Gamma_Q'$ in place of $\Gamma'$.  Define the subtree $\Gamma'$ of $\Gamma$ to be the root node of $\Gamma$ with the collection of at least $r^{\gthree}$ many children $Q$, each of those children followed by its subtree $\Gamma_Q'$.

We will now verify that (\ref{fertile_ancestry_with_ell_help}) holds for all nodes of $\Gamma'$. Let $Q$ be any node of $\Gamma'$. If $Q$ is the root node of $\Gamma'$, then (\ref{fertile_ancestry_with_ell_help}) holds because $\log_r \HC_r^{\gfour}(\Gamma) \leq 0$.  (Indeed, that $\HC_r^{\gfour}(\Gamma) \leq 1$ follows by considering the cut $\CC \defeq \{Q\}$ of $\Gamma$.) If $Q$ is a non-root node of $\Gamma'$, then it belongs to one of the subtrees $\Gamma_S'$ for some child $S$ of the root of $\Gamma'$. By property (\ref{fertile_ancestry_with_ell_help}) for the subsubtree $\Gamma_S'$, we see
\begin{align*}
|\FC_{\Gamma',r^{\gthree}}(Q)| - 1 &= |\FC_{\Gamma'_S,r^{\gthree}}(Q)| \\
&\geq \frac{|\mathcal{A}_{\Gamma'_S}(Q)|B + \log_r \HC_r^{\gfour}(\Gamma_S)}{A+B} \\
&\geq \frac{(|\mathcal{A}_{\Gamma'}(Q)|-1)B + \log_r \HC_r^{\gfour}(\Gamma) - A}{A+B}.
\end{align*}
This simplifies to the inequality in (\ref{fertile_ancestry_with_ell_help}), verifying the inductive step if Case \ref{tree_dichotomy_case_I} of \cref{lemma_tree_dichotomy_for_capacity} holds.

Suppose Case \ref{tree_dichotomy_case_II} of \cref{lemma_tree_dichotomy_for_capacity} holds. Let $Q$ be the child guaranteed by Case \ref{tree_dichotomy_case_II}. By the induction hypothesis, there exists a subtree $\Gamma_Q'$ of $\Gamma_Q$ in which every node satisfies (\ref{fertile_ancestry_with_ell_help}) with $\Gamma_Q$ in place of $\Gamma$ and $\Gamma_Q'$ in place of $\Gamma'$.  Define the subtree $\Gamma'$ of $\Gamma$ to be the root of $\Gamma$ with only the child $Q$ followed by its subtree $\Gamma_Q'$.

We will now verify that (\ref{fertile_ancestry_with_ell_help}) holds for all nodes of $\Gamma'$. Let $Q$ be any node of $\Gamma'$. If $Q$ is the root node of $\Gamma'$, then (\ref{fertile_ancestry_with_ell_help}) holds because $\log_r \HC_r^{\gfour}(\Gamma) \leq 0$.  If $Q$ is a non-root node of $\Gamma'$, then by property (\ref{fertile_ancestry_with_ell_help}) for the subtree containing $Q$, we see
\[|\FC_{\Gamma',r^{\gthree}}(Q)| \geq \frac{(|\mathcal{A}_{\Gamma'}(Q)|-1)B + \HC_r^{\gfour}(\Gamma)+B}{A+B}.\]
This simplifies to the inequality in (\ref{fertile_ancestry_with_ell_help}), verifying the inductive step if Case \ref{tree_dichotomy_case_II} of \cref{lemma_tree_dichotomy_for_capacity} holds. The proof of the inductive step is complete, and the lemma follows.
\end{proof}

\begin{theorem}
\label{tree_thinning_theorem}
For all $0 < \eps < 1$, for all $0 < \gthree < \gfour < \gfive < \gfour + \eps(\gfour - \gthree)$, for all sufficiently large $r \in \N$, and for all $V > 0$, there exists $N_0 \in \N$ for which the following holds.  For all $N \geq N_0$ and for all trees $\Gamma$ of height $N$ with $\HC_r^{\gfour}(\Gamma) \geq V$ that satisfy (\ref{soft_one_child_policy}), there exists a subtree $\Gamma'$ of $\Gamma$ such that all nodes $Q \in \Gamma'$ with height at least $N_0$ have $(r^{\gthree},1-\eps)$-fertile ancestry in $\Gamma'$.
\end{theorem}

\begin{proof}
Let $0 < \eps < 1$ and $0 < \gthree < \gfour < \gfive < \gfour + \eps(\gfour - \gthree)$.  Let $r \in \N$ be sufficiently large so that $\gfour - \gthree - \log_{r}2 > (1-\eps)(\gfive - \gthree)$. Define $A = \gfive - \gfour + \log_r 2$ and $B = \gfour - \gthree - \log_r 2$, and note by the inequality in the previous sentence, $B / (A+B) > (1-\eps)$. Let $V > 0$. Choose $N_0 \in \N$ such that
\begin{align}\label{key_inequality_in_tree_thinning_theorem}
    \frac{N_0B + \log_r V}{N_0(A+B)} > 1-\eps,
\end{align}
and note that for all $N \geq N_0$, the inequality in (\ref{key_inequality_in_tree_thinning_theorem}) holds with $N_0$ replaced by $N$.

Let $N \geq N_0$, and let $\Gamma$ be a tree of height $N$ with $\HC_r^{\gfour}(\Gamma) \geq V$ that satisfies (\ref{soft_one_child_policy}).  By \cref{pre_tree_thinning_theorem}, there exists a subtree $\Gamma'$ of $\Gamma$ such that for all nodes $Q$ of $\Gamma'$, the inequality in  (\ref{fertile_ancestry_with_ell_help}) holds.

Let $Q$ be a node of $\Gamma'$ with height at least $N_0$. By (\ref{fertile_ancestry_with_ell_help}) and (\ref{key_inequality_in_tree_thinning_theorem}), we see that
\[\frac{|\FC_{\Gamma',r^{\gthree}}(Q)|}{|\mathcal{A}_{\Gamma'}(Q)|} \geq \frac{|\mathcal{A}_{\Gamma'}(Q)|B + \log_r V}{|\mathcal{A}_{\Gamma'}(Q)|(A+B)} > 1-\eps.\]
It follows that $Q$ has $(r^{\gthree},1-\eps)$-fertile ancestry in $\Gamma'$, as was to be shown.
\end{proof}

\section{Proof of the sumsets theorem}
\label{sec_proof_of_theorem}

In this section, we prove \cref{main_real_theorem_intro}, the main theorem in this work.  We restate it here for the reader's convenience.

\begin{named}{Theorem A}{}
Let $r$ and $s$ be multiplicatively independent positive integers, and let $X, Y \subseteq [0,1]$ be $\times r$- and $\times s$-invariant sets, respectively. Define $\ogamma = \min \big ( \dimh X + \dimh Y, 1 \big)$. For all compact $I \subseteq \R \setminus \{0\}$ and all $\gamma < \ogamma$,
\begin{align}
\label{eqn_proof_section_result}
    \inf_{\lambda, \eta \in I}\ \HC_{\inftycontent}^{\gamma} \big( \lambda X + \eta Y \big) > 0.
\end{align}
\end{named}

Several auxiliary results go into the proof: the discrete version of Marstrand's projection theorem in \cref{section_projections_and_discrete_marstrand}, the subtree regularity theorem for finite trees in \cref{tree_section}, and the quantitative equidistribution result in \cref{section_quant_equid}.  We outline the proof of \cref{main_real_theorem_intro} in \cref{section_outline_of_proof} before presenting the full details in Sections \ref{section_proof_of_main_thm} and \ref{section_proof_of_main_claim}.

\begin{remark}
\label{remark_quantifyingA}
It is natural to ask about the value of the infimum $C \defeq \inf_{\lambda, \eta \in I}\ \HC_{\inftycontent}^{\gamma} \big( \lambda X + \eta Y \big)$ that appears in \eqref{eqn_proof_section_result}, or, more precisely, how it depends on $X$ and $Y$.
The value of $C$ must depend on $r$, $s$, $\gamma$, $\ogamma$, and $I$, but also on $X$ and $Y$, at least to the extent that it accounts for the Hausdorff content of $X\times Y$.
It follows from the proof of Theorem A below that this is essentially the only sense in which $C$ depends on $X$ and $Y$.
 
More precisely, there exist $\gfour,\gfive>0$ (depending only on $\gamma$ and $\dimh(X\times Y)$) with $\gfour < \dimh(X\times Y) < \gfive$ such that taking $M_0\in\N$ and $c_1,c_2>0$ as given by \cref{product_of_discrete_invariant_sets_bounds} when applied with $\gfive$ as $\xi$,
the quantity $C$
depends only on $M_0$, $c_1$, $c_2$, $r$, $s$, $I$, $\gamma$, $\dimh(X\times Y)$, and $\HC_{>0}^{\gfour}(X\times Y)$, but otherwise not on $X$ and $Y$.\footnote{As will be evident from the proof, our argument requires from $\gfour$ only that $\HC_{>0}^{\gfour}(X\times Y) > 0$.  Appealing to a result such as \cite[Theorem 3.1]{falconer_techniques_book_1997}, it is possible to take $\gfour=\dimh(X\times Y)$.  This improvement would eliminate a parameter but ultimately would not simplify the argument.}
\end{remark}

\subsection{Outline of the proof of \cref{main_real_theorem_intro}}\label{section_outline_of_proof}

Before beginning with the details of the proof of \cref{main_real_theorem_intro}, we explain the main ideas behind it. To understand the argument, it helps to begin by assuming that the set $X \times Y$ is self-similar in the sense that for every $n \in \Nz$, it is a union of approximately $r^{n (\dimh X + \dimh Y)}$ many translates of the set $r^{-n} X \times s^{-n'} Y$.  (Recall that $n' = \lfloor n \log r / \log s \rfloor$ so that $s^{-n'}\approx r^{-n}$.) This is the case, for example, if $X$ and $Y$ are both restricted digit Cantor sets. In this case, Peres and Shmerkin \cite{peresshmerkin2009} proved that for all $\lambda, \eta \in \R \setminus \{0\}$, $\dimh (\lambda X + \mu Y) = \ogamma$. Our argument follows along the same lines as theirs.

Recall that $\Pi_{t}: \R^2 \to \R$ is the oblique projection $\Pi_{t}(x,y) = x + t y$. A quick calculation shows that
\[\Pi_{e^t} (r^{-n} X \times s^{-n'} Y) = r^{-n} \Pi_{e^t r^n / s^{n'}} (X \times Y),\]
which implies that the images of the translates of $r^{-n} X \times s^{-n'} Y$ under the map $\Pi_{e^t}$ are affinely equivalent to the image of the full set $X \times Y$ under the map $\Pi_{e^t r^n / s^{n'}}$. It follows that the set $\Pi_{e^t} (X \times Y)$ contains affine images of the sets $\Pi_{e^t r^n / s^{n'}} (X \times Y)$ and hence that
\[\dimh \Pi_{e^t} (X \times Y) \geq \sup_{n \in \Nz} \dimh \Pi_{e^t r^n / s^{n'}} (X \times Y).\]
Thus, to bound $\dimh \Pi_{e^t} (X \times Y)$ from below, it suffices to show that there is some $n \in \Nz$ for which $e^t r^n / s^{n'}$ is a ``good angle'' for $X \times Y$, in the sense that $\dimh \Pi_{e^t r^n / s^{n'}} (X \times Y) > \ogamma - \eps$. It follows from Marstrand's theorem that the set of such ``good angles'' for $X \times Y$ (indeed, for any set) has full measure in $\R$, and it will be shown that the sequence $n \mapsto \log(e^t r^n / s^{n'})$ has image in $[t,t+\log s)$ and is the orbit of $t$ under the irrational $x \mapsto x + \log r \pmod{\log s}$ translated by $t$. When combined, these facts fall just short of allowing us to conclude the existence of $n \in \Nz$ for which $e^t r^n / s^{n'}$ is a good angle: it is possible that the image of an equidistributed sequence misses a set of full measure.

To make use of the above outline, one needs to gain some topological information on the set of good angles from Marstrand's theorem. This can be accomplished by moving the argument to a discrete setting. Discretizing introduces a number of technical nuisances, but the core of the argument remains the same. Recall that $X_n$ and $Y_{n'}$ are the sets $X$ and $Y$ rounded to the lattices $r^{-n} \Z$ and $s^{-n'} \Z$, respectively. The discrete analogue of Marstrand's theorem in \cref{thm:discretemarstrandbase} tells us that the complement of the set of ``good angles'' for a finite set such as $X_n \times Y_{n'}$ can be covered by a disjoint union of few half-open intervals.  This topological information combines with the equidistribution of the irrational rotation described above to allow us to find many $n \in \Nz$ for which $e^t r^n / s^{n'}$ is a good angle for $X_n \times Y_{n'}$.

The argument described thus far is essentially due to Peres and Shmerkin in \cite{peresshmerkin2009} and allows them to conclude that for all $t \in \R \setminus \{0\}$, $\dimh \Pi_{e^t} (X \times Y) = \ogamma$.  We will now describe the two primary modifications we make to this argument in the course of the proof of \cref{main_real_theorem_intro}.

The first modification allows us to show that the discrete Hausdorff content of $\Pi_{e^t} (X \times Y)$ at all small scales is uniform in $t$.  Ultimately, this uniformity stems from the fact that the irrational rotation described above is uniquely ergodic: changing $t$ in the argument above changes only the point whose orbit we consider.  Exposing the uniformity in the argument after this is then mainly a matter of taking care with the quantifiers in the auxiliary results.

The second modification allows us to handle sets $X$ and $Y$ which are only assumed to be $\times r$- and $\times s$-invariant. Such sets need not be self-similar, but they do exhibit some ``near self similarity'' in the following sense.  Consider the discrete set $X_m$ for some large $m \in \N$.  Because $X$ is $\times r$-invariant, the set $X_{(n+1)m} \cap \big[i /  r^{nm}, (i + 1)/ r^{nm} \big)$, when dilated by $r^{mn}$ and considered modulo $1$, is a subset of $X_m$.  While this set is generally not equal to $X_m$, it is, by an averaging argument, very often of cardinality greater than $r^{-\eps} |X_m|$. This is profitably re-interpreted in the language of trees: in the tree with levels $X_{nm} \times Y_{(nm)'}$, $n \in \Nz$, many nodes have nearly the maximum allowed number of children.  The tree thinning result in \cref{tree_thinning_theorem} exploits this abundance by finding a sufficiently ``regular'' subtree on which we focus our attention.  Then, we invoke our discrete analogue of Marstrand's theorem -- which provides information on the set of angles that are good not only for the original set $X_m \times Y_{m'}$, but also for large subsets of it -- to further thin the subtree. Following the reasoning given in \cref{remark_reason_for_fertile_ancestry}, the resulting subtree has fertile ancestry and hence has large Hausdorff content.  By the construction of the subtree, its image under $\Pi_{e^t}$ is large, and this yields the lower bound on the Hausdorff dimension in the conclusion of the theorem.

\subsection{Proof of \cref{main_real_theorem_intro}}\label{section_proof_of_main_thm}

In this section and the next, let $r$, $s$, $X$, $Y$, and $\ogamma$ be as given as in the statement of \cref{main_real_theorem_intro}.  The proof of \cref{main_real_theorem_intro} begins with a number of reductions, the last of which in \cref{theorem_uniform_hausdorff_dimension_projections_second_form} is a statement about the existence of measures on the images of the discrete product sets under oblique projections.  We prove \cref{theorem_uniform_hausdorff_dimension_projections_second_form} in the next subsection.

By \cref{dimension_of_product_of_invariant_sets}, $\dimh ( X \times Y) = \dimh X + \dimh Y$. Note that if $\dimh X = 0$, then the conclusion is clear by considering, for any $x \in X$, images of the set $\{x\} \times Y$. The same is true if $\dimh Y=0$. Thus, we will proceed under the assumption that $\dimh X, \dimh Y > 0$.  Note that the set $1-X$ is $\times r$-invariant and that $-\lambda X + \eta Y$ is a translate of the set $\lambda(1-X) + \eta Y$.  The analogous statement holds for $Y$.  Combining these facts, it is easy to see that it suffices to prove \cref{main_real_theorem_intro} in the case that $I \subseteq (0,\infty)$.

The next step is to formulate a statement sufficient to prove \cref{main_real_theorem_intro} in terms of oblique projections of discrete sets. Recall that $n' = \lfloor n \log  r / \log s\rfloor$ and that $X_n$, $Y_{n'}$ are the sets $X$ and $Y$ rounded to the lattices $r^{-n} \Z$ and $s^{-n'} \Z$, respectively. For $n \in \Nz$, define
\begin{align}\label{def_of_product_sets}
    \QC_n = X_n \times Y_{n'} \qquad \text{ and } \qquad \QCt_n = X_n \times Y_{n'+1}.
\end{align}

\begin{claim}\label{theorem_uniform_hausdorff_dimension_projections_second_form}
For all compact $I \subseteq \R$ and all $0 < \gamma < \ogamma$, there exists $m, N_0 \in \N$ such that for all $N \geq N_0$ and all $t \in I$, there exists a probability measure $\mu$ supported on the finite set $\Pi_{e^{t}} \QC_{Nm}$ with the property that for all balls $B \subseteq \R$ of diameter $\delta \geq r^{-Nm}$, $\mu(B) \leq r^{N_0m} \delta^{\gamma}$.
\end{claim}

To deduce \cref{main_real_theorem_intro} from \cref{theorem_uniform_hausdorff_dimension_projections_second_form}, let $I \subseteq (0, \infty)$ be compact and $0 < \gamma < \ogamma$. Apply \cref{theorem_uniform_hausdorff_dimension_projections_second_form} with $\tilde I \defeq \big\{\log(\eta/\lambda) \ \big| \ \eta,\lambda\in I \big\}$ as $I$ and $\gamma$ as it is.  Let $m, N_0 \in \N$ be as guaranteed by \cref{theorem_uniform_hausdorff_dimension_projections_second_form}.

Note that by \cref{lemma_connection_between_discrete_and_unlimited_H_contents} and the fact that the function $\rho \mapsto \HC_{\geq \rho}^{\gamma} \allowbreak \big( \lambda X + \eta Y \big)$ is non-increasing (as $\rho$ decreases),
\[\inf_{\lambda, \eta \in I}\ \HC_{\inftycontent}^{\gamma} \big( \lambda X + \eta Y \big)  = \lim_{\rho \to 0} \inf_{\lambda, \eta \in I}\ \HC_{\geq \rho}^{\gamma} \big( \lambda X + \eta Y \big).\]
The limit in the final expression exists because $\inf_{\lambda, \eta \in I}\ \HC_{\geq \rho}^{\gamma} \big( \lambda X + \eta Y \big)$ is non-increasing and is bounded from below by zero.

Therefore, to show that \eqref{eqn_proof_section_result} holds, it suffices to prove that
\begin{align}\label{lower_bound_on_hausdorff_content_in_main_theorem_two}
    \lim_{N \to \infty}\ \inf_{\lambda, \eta \in I}\ \HC_{\geq r^{-Nm}}^{\gamma} \big( \lambda X + \eta Y \big) > 0.
\end{align}
It follows from the fact that
\[d_H(\lambda X_{Nm} + \eta Y_{(Nm)'}, \lambda X + 
\eta Y) \ll_{I, r, s} r^{-Nm},\]
where $d_H$ is the Hausdorff metric, and \cref{lem_box_counting_estimate} that for all $\lambda, \eta \in I$,
\begin{align}\label{bounds_on_hausdorff_content_in_proof_reduction}
    \begin{aligned}\HC_{\geq r^{-Nm}}^{\gamma} \big( \lambda X + \eta Y \big) &\asymp_{I,r,s} \HC_{\geq r^{-Nm}}^{\gamma} \big(\lambda X_{Nm} + \eta Y_{(Nm)'} \big)\\
    &\asymp_{I,r,s} \HC_{\geq r^{-Nm}}^{\gamma} \big( X_{Nm} + e^{\log(\eta / \lambda)} Y_{(Nm)'} \big).\end{aligned}
\end{align}
Therefore, to show \eqref{lower_bound_on_hausdorff_content_in_main_theorem_two}, it suffices to prove that
\begin{align}\label{lower_bound_on_hausdorff_content_in_main_theorem_three}
    \lim_{N \to \infty}\ \inf_{t \in \tilde I}\ \HC_{\geq r^{-Nm}}^{\gamma} \big( \Pi_{e^{t}} \QC_{Nm} \big) > 0.
\end{align}
Combining the conclusion of \cref{theorem_uniform_hausdorff_dimension_projections_second_form} with \cref{lem_consequence_of_discrete_frostman}, we see that for all $N \geq N_0$ and $t \in \tilde I$, $\HC_{\geq r^{-Nm}}^{\gamma} \big(\Pi_{e^{t}} \QC_{Nm} \big) \geq r^{-N_0m}$.  This shows that the limit in \eqref{lower_bound_on_hausdorff_content_in_main_theorem_three} is positive and completes the deduction of \cref{main_real_theorem_intro} from \cref{theorem_uniform_hausdorff_dimension_projections_second_form}.

\subsection{Proof of \cref{theorem_uniform_hausdorff_dimension_projections_second_form}}\label{section_proof_of_main_claim}

\textbf{Choosing the parameter $m$ and scale $\rho$.} \quad Recall that $r$, $s$, $X$, $Y$, and $\ogamma$ are as given as in the statement of \cref{main_real_theorem_intro}. Without loss of generality, we can assume that $r < s$.  Put $\beta = \log s$, let $0< \gamma < \ogamma$, and define $\eps \defeq \ogamma - \gamma$ and $\gone\defeq  \gamma$. 

We claim that there exist $\gtwo$, $\gthree$, $\gfour$, and $\gfive$ such that
\begin{enumerate}[label=(\Roman*)]
    \item \label{gamma_inequality_one} $0 < \gone / (1-\eps/2) < \gtwo < \gthree < \gfour< \dimh X + \dimh Y < \gfive$;
    \item \label{gamma_inequality_two} $\gfive < \gfour + \eps (\gfour - \gthree) / 6$;
    \item \label{gamma_inequality_three} $\gtwo < 1$;
    \item \label{gamma_inequality_four} $2(\gfive - \gthree) < \gfour - \gtwo$ (this is the inequality in (\ref{inequality_between_gammas_for_marstrand})).
    \newcounter{tryhard}
    \setcounter{tryhard}{\arabic{enumi}}
\end{enumerate}
To see why, note that if we put $\gtwo = \gone / (1-\eps/2)$, $\gfive = \gfour = \dimh X + \dimh Y$, and $\gthree = \gtwo / 3 + 2 \gfour / 3$, then the inequalities in \ref{gamma_inequality_one} holds with ``$<$'' replaced by ``$\leq$'', while the inequalities in \ref{gamma_inequality_two}, \ref{gamma_inequality_three}, and \ref{gamma_inequality_four} hold as written.  It follows that $\gtwo$ and $\gfive$ can be increased and $\gfour$ can be decreased (with the corresponding change to $\gthree = \gtwo / 3 + 2 \gfour / 3$) so that all of the strict inequalities hold.

Let $c_1$, $c_2$, and $M_0$ be the constants guaranteed by \cref{product_of_discrete_invariant_sets_bounds}, when applied with $\gfive$ as $\xi$.
Let $I \subseteq (0,\infty)$ be compact, and define $I_\beta = I + [0,\beta]$. 
Let $P > 0$ be a Lipschitz constant for all of the maps $\Pi_{e^t}$, $t \in I_\beta$, and let $c_3=4Ps^{-1}+1$. Choose $m\in\N$ large enough so that we can apply
\begin{itemize}
    \item \cref{tree_thinning_theorem} with $\eps / 6$ as $\eps$ and $r^m$ as $r$;
    \item \cref{thm:discretemarstrandcorollary} with $I_\beta$ as $I$, $\epsilon\beta/12$ as $\epsilon$ and $r^{-m}$ as $\rho$;
    \item \cref{product_of_discrete_invariant_sets_bounds} with $m$ as $N$ (i.e., $m \geq M_0$). 
\end{itemize}
Put $\rho = r^{-m}$.
\\

\noindent\textbf{A uniformly distributed sequence.} \quad Let $\al =\log \big(r^m / s^{m'}\big)$ and let $R:[0,\beta)\to[0,\beta)$ be the transformation $R:x \mapsto x + \al \pmod\beta$.
As $\beta=\log s$ and $m'=\lfloor m\log r/\log s\rfloor$, we have \begin{align}\label{eq_alphabeta}
    \alpha/\beta=m\log r/\log s-m'=\big\{m\log r/\log s \big\}. 
\end{align}
Since $\log r/\log s$ is irrational, we conclude that  $\al/\beta$ is irrational, whereby the sequence $(R^n(0))_{n \in \Nz}$ is uniformly distributed on $[0,\beta)$.

\begin{claim} For all $n \in \Nz$,
\begin{enumerate}[label={\normalfont(\Roman*)}]
\setcounter{enumi}{\arabic{tryhard}}
\item \label{item:descriptionofrotation} $R^n(0)+(nm)'\log s= nm\log r $; 
\eqnitem\label{itemVII} 
\[\big((n+1)m\big)'=\begin{cases} (nm)'+m'&\text{ if }R^n(0)+\alpha<\beta\\
(nm)'+m'+1&\text{ if }R^n(0)+\alpha>\beta\end{cases}.\]
\setcounter{tryhard}{\arabic{enumi}}
\end{enumerate}
\end{claim}

\begin{proof}
Since for all $n\in\N$, $R^n(0)=n\alpha\pmod{\beta}$, using \eqref{eq_alphabeta}, we can write $R^n(0)/\beta=\big\{n\alpha/\beta\big\}=\big\{n\{m\log r/\beta\}\big\}=\big\{nm\log r/\beta\big\}$. 
Recalling that $(nm)'=\lfloor nm\log r/\beta\rfloor$, this establishes \ref{item:descriptionofrotation}.

Next, note that for any real numbers $x,y$,
$$\lfloor x+y\rfloor=\begin{cases}
\lfloor x\rfloor+\lfloor y\rfloor&\text{ if }\{x\}+\{y\}<1
\\
\lfloor x\rfloor+\lfloor y\rfloor+1&\text{ if }\{x\}+\{y\}\geq1
\end{cases}.$$
The equality in \ref{itemVII} follows from this by substituting $x=nm\log r/\beta$ and $y=m\log r/\beta$ and using $R^n(0)/\beta=\big\{nm\log r/\beta \big\}$ and \eqref{eq_alphabeta}.
\end{proof}

\noindent\textbf{Choosing the parameter $N_0$.} \quad From \cref{product_of_discrete_invariant_sets_bounds}, the sets $\QC_{m}$ and $\QCt_{m}$ (defined in \eqref{def_of_product_sets}) are $(c_1\rho,\gfive)_{c_2}$-sets and satisfy
\begin{align}\label{size_bounds_on_discrete_products}
    \rho^{-\gfour} \leq |\QC_{m}|, |\QCt_{m}| \leq \rho^{-\gfive}.
\end{align}
Let $T_1$ (resp. $T_2$) be the subset of $I_\beta$ obtained from applying \cref{thm:discretemarstrandcorollary} with $I_\beta$ as $I$, $\eps \beta / 12$ as $\eps$ and $\QC_m$ (resp. $\QCt_m$) as $A$. Put $T = T_1 \cap T_2$. It follows from \cref{thm:discretemarstrandcorollary} that $I_\beta \setminus T$ is covered by a disjoint union of not more than $U \defeq \lceil 2 \eps \beta \rho^{-1} / 24 \rceil$ many half-open intervals of Lebesgue measure less than $\eps \beta / 6$.

Let $N_0 \in \N$ be the larger of
\begin{itemize}
\item the $N_0$ from \cref{tree_thinning_theorem} with $\eps / 6$ as $\eps$, $r^m$ as $r$, and $2^{-\gfour} \HC^{\gfour}_{\inftycontent}(X \times Y)$ as $V$; 
\item the $N_0$ from \cref{cor:higherunifdistupdated} with $(R^n(0))_{n \in \Nz}$ as $(x_n)_{n \in \Nz}$ and $\eps / 3$ as $\eps$.
\end{itemize}

~

\noindent\textbf{Fixing the parameters $N$ and $t$.} \quad To prove \cref{theorem_uniform_hausdorff_dimension_projections_second_form}, we will show that for all $N \geq N_0$ and all $t \in I$ there exists a probability measure $\mu$ supported on the set $\Pi_{e^{t}} \QC_{Nm}$ with the property that for all balls $B \subseteq \R$ of diameter $\delta \geq \rho^N$, $\mu(B) \leq \rho^{-N_0} \delta^{\gamma}$.
Let $N \geq N_0$ and $t \in I$. From this point on, all new quantities and objects can depend on $N$ and $t$.\\

\noindent\textbf{Constructing the tree $\Gamma$.} \quad Let $\Gamma$ be the tree (see \cref{def_tree}) of height $N$ with node set at height $n \in \{0, 1, \ldots, N\}$ equal to $\QC_{nm}$.
Associating the point $(i/ {r^{mn}}, j/s^{(mn)'}) \in \QC_{nm}$ with the rectangle
\[\left[\frac i{r^{mn}},\frac{i+1}{r^{mn}}\right)\times\left[\frac{j}{s^{(mn)'}},\frac{j+1}{s^{(mn)'}}\right),\]
parentage in the tree $\Gamma$ is determined by containment amongst associated rectangles.
Denote by $C_\Gamma(Q)$ the children of the node $Q$ in $\Gamma$. Denote by $\odot: \R^2 \times \R^2 \to \R^2$ the binary operation of pointwise multiplication.
\begin{claim}\leavevmode
Let $n<N$ and $Q\in\QC_{nm}$.
\begin{enumerate}[label={\normalfont (\Roman*)},leftmargin=*] \setcounter{enumi}{\arabic{tryhard}}
    \item \label{item:caseonenode} If $R^n(0) + \alpha < \beta$, then     $C_\Gamma(Q) \subseteq Q + (r^{-nm}, s^{-(nm)'}) \odot \QC_{m}$. 
    \item \label{item:casetwonode} If $R^n(0) + \alpha > \beta$, then     $C_\Gamma(Q) \subseteq Q + (r^{-nm}, s^{-(nm)'}) \odot \QCt_{m}$. 
    \item\label{item_hausdorffcontent} $\HC_{r^m}^{\gfour}(\Gamma) \geq 2^{-\gfour}\HC^{\gfour}_{\inftycontent}(X \times Y)$.\\
    \setcounter{tryhard}{\arabic{enumi}}
\end{enumerate}
\end{claim}

\begin{proof}
    We first prove parts \ref{item:caseonenode} and \ref{item:casetwonode}. 
    By \cref{mult_invariance_of_approximations}, $r X_n \subseteq X_{n-1} \pmod 1$ and $s Y_{n'} \subseteq Y_{n'-1} \pmod 1$. 
    By \ref{itemVII}, if $R^n(0) + \alpha < \beta$, then $\big((n+1)m \big)'=(nm)' + m'$, and hence $(r^{nm}, s^{(nm)'}) \odot \QC_{(n+1)m} \subseteq \QC_{m} \pmod 1$, and in particular $(r^{nm}, s^{(nm)'}) \odot C_\Gamma(Q) \subseteq \QC_{m} \pmod 1$. 
    If $R^n(0) + \alpha > \beta$, then $\big((n+1)m \big)'=(nm)' + m'+1$, and hence $(r^{nm}, s^{(nm)'}) \odot \QC_{(n+1)m} \subseteq \QCt_{m} \pmod 1$, and in particular $(r^{nm}, s^{(nm)'}) \odot C_\Gamma(Q) \subseteq \QCt_{m} \pmod 1$. 
    
    Write $Q = (i / r^{nm}, j / s^{(nm)'})$ and let $Q' \in C_\Gamma(Q)$. 
    Because $Q'$ is a child of $Q$, we can write $Q' = Q + (i_0/r^{(n+1)m},j_0/s^{((n+1)m)'})$ where $0 \leq i_0 < r^m$ and $0 \leq j_0 < s^{m'}$. 
    It follows that $(r^{nm}, s^{(nm)'}) \odot (C_\Gamma(Q)-Q) \subseteq \QC_{m}$ (in the first case $R^n(0) + \alpha < \beta$) or $(r^{nm}, s^{(nm)'}) \odot (C_\Gamma(Q)-Q) \subseteq \QCt_{m}$ (in the second case $R^n(0) + \alpha > \beta$), where the containment now is understood without reducing modulo 1.

To prove \ref{item_hausdorffcontent}, take a cut $\{Q_1, \ldots, Q_\ell\} \subseteq \Gamma$ of $\Gamma$ with node $Q_i$ at height $n_i$. 
Then, by construction of $\Gamma$, there exists a cover $X \times Y \subseteq \cup_{i=1}^\ell B_i$ where ball $B_i$ has diameter at most $2 \rho^{n_i}$.  Since the cut was arbitrary, it follows that $\HC_{r^m}^{\gfour}(\Gamma) \geq 2^{-\gfour}\HC^{\gfour}_{\inftycontent}(X \times Y)$.
\end{proof}

\noindent\textbf{Constructing the tree $\Gamma'$.} \quad Combining \eqref{size_bounds_on_discrete_products} with \ref{item:caseonenode} and \ref{item:casetwonode}, it follows that $|C_\Gamma(Q)|\leq r^{m\gfive}$ for every non-leaf node $Q$ of $\Gamma$.
The tree $\Gamma$ has now been shown to satisfy all the hypothesis of \cref{tree_thinning_theorem} (with $\eps / 6$ as $\eps$, $r^m$ as $r$, and $2^{-\gfour} \HC^{\gfour}_{\inftycontent}(X \times Y)$ as $V$), thus there exists a subtree $\Gamma'$ of $\Gamma$ with the property that every node with height at least $N_0$ has $(r^{m \gthree},1-\eps / 6)$-fertile ancestry in $\Gamma'$.\\

\noindent\textbf{Constructing the tree $\Gamma''$.} \quad Now we will use \cref{thm:discretemarstrandcorollary}, the corollary to the discrete version of Marstrand's theorem, to further thin out the tree $\Gamma'$; an outline for this step was described in \cref{remark_reason_for_fertile_ancestry} \ref{reason_for_fertile_ancestry_one}. 
For each non-leaf node $Q \in \Gamma'$, we will define a subset $C_{\Gamma'}^m(Q)$ of $C_{\Gamma'}(Q)$.  
Define $J = (T - t) \cap [0, \beta)$.  
Since $I_\beta \setminus T$ is covered by at most $U$ many half-open intervals of measure less than $\eps \beta / 6$, the same is true for the set $[0,\beta) \setminus J$. Define $\JC = \{ 0 \leq n \leq N-1 \ | \ R^n(0) \in J\}$. Note that for all $n \geq N_0$, by \cref{cor:higherunifdistupdated}, $|\JC \cap \{0, \ldots, n-1\}| \geq (1-\eps / 3)n$.

Let $Q$ be a non-leaf node of $\Gamma'$, and let $n \in \{0, \ldots, N-1\}$ be the height of $Q$. Consider the following cases:
\begin{enumerate}[label=(\Roman*)] \setcounter{enumi}{\arabic{tryhard}}
    \item $n \not\in \JC$ or $|C_{\Gamma'}(Q)| < \rho^{-\gthree}$. Select a single child $Q'$ of $Q$ and put $C_{\Gamma'}^m(Q) = \{Q'\}$.
    \item \label{marstrand_applied_to_case_one_node} $n \in \JC$, $|C_{\Gamma'}(Q)| \geq \rho^{-\gthree}$, and $R^n(0) + \al < \beta$. 
    By \cref{tree_thinning_theorem} and \ref{item:caseonenode}, the set $A' \defeq (r^{nm}, s^{(nm)'}) \odot (C_{\Gamma'}(Q) - Q)$ is a subset of $\QC_{m}$ of cardinality at least $\rho^{-\gthree}$. 
    Since $n \in \JC$, we have that $t+ R^n(0) \in T$. 
    Applying \cref{thm:discretemarstrandcorollary} \ref{marstrand_corollary_conclusion_three} with $t+ R^n(0) $ in the role of $t$, there exists a subset $A'_t \subseteq A'$ with $|A'_t| \geq \rho^{-\gtwo}$ and such that the points of $\Pi_{e^{t+ R^n(0)}} A'_t$ are distinct and $c_3 \rho$-separated. Define $C_{\Gamma'}^m(Q) = Q + (r^{-nm}, s^{-(nm)'}) \odot A'_t$ so that $(r^{nm}, s^{(nm)'}) \odot (C_{\Gamma'}^m(Q) - Q) = A'_t$.
    \item \label{marstrand_applied_to_case_two_node} $n \in \JC$, $|C_{\Gamma'}(Q)| \geq \rho^{-\gthree}$, and $R^n(0) + \al > \beta$. We do exactly as in \ref{marstrand_applied_to_case_one_node} with $\mathcal{Q}_m$ replaced by $\QCt_{m}$ and using \ref{item:casetwonode} to get the set $C_{\Gamma'}^m(Q)$.
    \setcounter{tryhard}{\arabic{enumi}}
\end{enumerate}
Let $\Gamma''$ be the subtree of $\Gamma'$ with the property that if $Q$ is a non-leaf node of $\Gamma''$, then $C_{\Gamma''}(Q) = C_{\Gamma'}^m(Q)$.  We claim that
\begin{align}\label{main_claim_about_gamma_prime_prime}
    \text{every node of $\Gamma''$ with height at least $N_0$ has $(r^{m\gtwo}, 1-\eps/2)$-fertile ancestry.}
\end{align} 
Indeed, let $Q$ be a node of $\Gamma''$ with height $n \geq N_0$. 
The ancestry of $Q$ in $\Gamma'$ is $(r^{m\gthree}, 1-\eps/6)$-fertile.  
Each $r^{m\gthree}$-fertile ancestor of $Q$ in $\Gamma'$ with height in the set $\JC$ is an $r^{m\gtwo}$-fertile ancestor of $Q$ in $\Gamma''$. Since $\big|\JC \cap \{0, \ldots, n-1\} \big| \geq (1-\eps/3)n$, there are at least $(1-\eps/2)n$ many $r^{m\gthree}$-fertile ancestor of $Q$ in $\Gamma'$ with height in the set $\JC$.
It follows that $Q$ has $(r^{m\gtwo}, 1-\eps/2)$-fertile ancestry in $\Gamma''$.

\begin{claim}\label{claim_metric_tree_morphism}
If $L_1$ and $L_2$ are two distinct leaves of $\Gamma''$ and $n$ is maximal such that $L_1$ and $L_2$ have a common ancestor at height $n$, then $|\Pi_{e^t} L_1 - \Pi_{e^t} L_2| \geq \rho^{n+1}$.
\end{claim}

\begin{proof}
Let $Q$ be the common ancestor of $L_1$ and $L_2$ in $\Gamma''$ of height $n$.
Note that by the definition of $\Gamma''$ and maximality of $n$, it must be that $Q$ has more than one child and hence that $n \in \JC$.
Let $Q_1$ and $Q_2$ be the children of $Q$ in $\Gamma''$ that are ancestors of $L_1$ and $L_2$, respectively.
Note that $Q_1 \neq Q_2$ but that $Q_i$ may be equal to $L_i$.

We will show first that $\Pi_{e^t} Q_1$ and $\Pi_{e^t} Q_2$ are $c_3\rho^{n+1}$-separated. 
Write $Q = (p,q)$ and $Q_i = (p_i,q_i)$. 
Suppose that $R^{n}(0) + \alpha < \beta$.  
It follows from \ref{item:descriptionofrotation} that
\begin{align}\label{skew_equality_for_oblique_projections}
\begin{aligned}\Pi_{e^t} Q_i &= r^{-nm} \left( r^{nm} \Pi_{e^t} (Q_i - Q)\right) + \Pi_{e^t} Q \\
&=\rho^{n} \left( r^{nm} (p_i - p) + e^{t + R^n(0)} s^{(nm)'} (q_i - q) \right) + \Pi_{e^t} Q \\
&= \rho^{n} \left( \Pi_{e^{t + R^n(0)}} \big((r^{nm},s^{(nm)'}) \odot (Q_i - Q) \big) \right) + \Pi_{e^t} Q.
\end{aligned}\end{align}
By \ref{marstrand_applied_to_case_one_node}, the points of $\Pi_{e^{t + R^n(0)}} \big((r^{nm},s^{(nm)'}) \odot (Q_i - Q) \big)$, $i=1,2$, are $c_3\rho$-separated. It follows then from (\ref{skew_equality_for_oblique_projections}) that the points of $\Pi_{e^t} Q_i$, $i=1,2$, are $c_3\rho^{n + 1}$-separated.  A similar argument works to reach the same conclusion if $R^n(0) + \alpha > \beta$ using \ref{marstrand_applied_to_case_two_node}.

By the definition of the $\QC_{nm}$ sets, $|Q_i - L_i| \leq 2s^{-1}\rho^{n+1}$. By the triangle inequality and the fact that $c_3 = 4Ps^{-1}+1$,
\begin{align*}
    \big|\Pi_{e^t} L_1 - \Pi_{e^t} L_2\big| &\geq \big|\Pi_{e^t} Q_1 - \Pi_{e^t} Q_2\big| - \big|\Pi_{e^t} (Q_1 - L_1)\big| - \big|\Pi_{e^t} (Q_1 - L_1)\big|\\
    &\geq (4Ps^{-1}+1)\rho^{n+1} - 4Ps^{-1}\rho^{n+1} \geq \rho^{n+1}.
\end{align*}
It follows that $\big|\Pi_{e^t} L_1 - \Pi_{e^t} L_2\big| \geq \rho^{n+1}$, as was to be shown.
\end{proof}

\noindent\textbf{Constructing the measure $\mu$.} \quad The proof of \cref{theorem_uniform_hausdorff_dimension_projections_second_form} will be concluded by demonstrating that 1) the fertile ancestry property of $\Gamma''$ in \eqref{main_claim_about_gamma_prime_prime} guarantees that $\Gamma''$ supports a ``measure'' which is not too concentrated on any node (an outline for this step was described in \cref{remark_reason_for_fertile_ancestry} \ref{reason_for_fertile_ancestry_two}); and 2) by \cref{claim_metric_tree_morphism}, the projection of this measure is not too concentrated on any ball.

Let $\nu: \Gamma'' \to [0,1]$ be the unique function that takes $1$ on the root of $\Gamma''$ and has the properties that for all non-leaf nodes $Q$ of $\Gamma''$, $\nu$ is constant on $\mathcal{C}_{\Gamma''}(Q)$ and $\nu(Q) = \sum_{C \in \mathcal{C}_{\Gamma''}(Q)} \nu(C)$. (Colloquially, a mass of 1 begins at the root of $\Gamma''$ and spreads down the tree by splitting equally amongst the children of each node.) Let $\nu_N$ be the function $\nu$ restricted to $\Gamma_N''$, the set of leaves of $\Gamma''$.  By the defining properties of $\nu$, the function $\nu_N$ is a probability measure on $\Gamma''_N$. 

Since $\Gamma''_N \subseteq \QC_{Nm}$, the measure $\mu = \Pi_{e^t} \nu_N$, the push-forward of $\nu_N$ through the map $\Pi_{e^t}$, is a probability measure supported on the set $\Pi_{e^{t}} \QC_{Nm}$.  We will conclude the proof of \cref{theorem_uniform_hausdorff_dimension_projections_second_form} by verifying that for all balls $B \subseteq \R$ of diameter $\delta \geq \rho^N$, $\mu(B) \leq \rho^{-N_0} \delta^{\gone}$. (Recall that $\gone = \gamma$.)

Let $B \subseteq \R$ be an interval of length $\delta \geq \rho^N$. 
Put $n = \lfloor \log_\rho \delta \rfloor + 1$ and note that $\rho^{n} < \delta \leq \rho^{n-1}$. 
It follows from \cref{claim_metric_tree_morphism} that there exists a node $Q$ of $\Gamma''$ with height at least $n$ with the property that if $L$ is a leaf of $\Gamma''$ with $\Pi_{e^t} L \in B$, then $Q$ is an ancestor of $L$.
This implies that $\mu(B) \leq \nu(Q)$, and so it suffices to show that
\begin{align}\label{small_mass_on_balls_in_proof}
    \nu(Q) \leq \rho^{-N_0} \delta^{\gone}.
\end{align}

If $n \leq N_0$, then $\rho^{-N_0} \delta^{\gone} > 1$ and (\ref{small_mass_on_balls_in_proof}) holds trivially.  If $n > N_0$, then by the definition of $\nu$ and the fact that $Q$ has $(r^{m\gtwo},1-\eps/2)$-fertile ancestry (cf. \eqref{main_claim_about_gamma_prime_prime}),
\[\nu(Q) \leq \frac{1}{r^{m\gtwo(1-\eps/2)n}} = \rho^{\gtwo(1-\eps/2)n} \leq \rho^{-N_0} \delta^{\gone},\]
since $(1-\eps/2) \gtwo > \gone$. This verifies (\ref{small_mass_on_balls_in_proof}), completing the proof of \cref{theorem_uniform_hausdorff_dimension_projections_second_form} and hence of \cref{main_real_theorem_intro}.

\bibliographystyle{alphanum}
\bibliography{integercantorbib}

\begin{thebibliography}{GMR}

\bibitem[Aus]{austin_proof_of_furstenberg_2020}
Tim Austin.
\newblock A new dynamical proof of the {S}hmerkin-{W}u theorem.
\newblock {\em J. Mod. Dyn.}, 18:1--11, 2022.

\bibitem[BP]{bishopperesbook}
Christopher~J. Bishop and Yuval Peres.
\newblock {\em Fractals in probability and analysis}, volume 162 of {\em
  Cambridge Studies in Advanced Mathematics}.
\newblock Cambridge University Press, Cambridge, 2017.

\bibitem[BY]{yu_digit_expansions_arxiv}
Stuart~A. Burrell and Han Yu.
\newblock Digit expansions of numbers in different bases.
\newblock {\em J. Number Theory}, 226:284--306, 2021.

\bibitem[Fal]{falconer_techniques_book_1997}
Kenneth Falconer.
\newblock {\em Techniques in fractal geometry}.
\newblock John Wiley \& Sons, Ltd., Chichester, 1997.

\bibitem[FFJ]{falconer_projection_survey}
Kenneth Falconer, Jonathan Fraser, and Xiong Jin.
\newblock Sixty years of fractal projections.
\newblock In {\em Fractal geometry and stochastics {V}}, volume~70 of {\em
  Progr. Probab.},  3--25. Birkh\"{a}user/Springer, Cham, 2015.

\bibitem[FO]{Fassler_Orponen_2014}
Katrin F\"{a}ssler and Tuomas Orponen.
\newblock On restricted families of projections in {$\Bbb R^3$}.
\newblock {\em Proc. Lond. Math. Soc. (3)}, 109(2):353--381, 2014.

\bibitem[Fra]{fraser_2020}
Jonathan~M. Fraser.
\newblock {\em Assouad Dimension and Fractal Geometry}.
\newblock Cambridge Tracts in Mathematics. Cambridge University Press, 2020.

\bibitem[Fur1]{furstenbergdisjointness}
Harry Furstenberg.
\newblock Disjointness in ergodic theory, minimal sets, and a problem in
  {D}iophantine approximation.
\newblock {\em Math. Systems Theory}, 1:1--49, 1967.

\bibitem[Fur2]{furstenbergtransversality}
Harry Furstenberg.
\newblock Intersections of {C}antor sets and transversality of semigroups.
\newblock  41--59, 1970.

\bibitem[Gla]{glasscock_marstrand_for_integers}
Daniel Glasscock.
\newblock Marstrand-type theorems for the counting and mass dimensions in
  {$\mathbb{Z}^d$}.
\newblock {\em Combin. Probab. Comput.}, 25(5):700--743, 2016.

\bibitem[GMR]{GMR_add_trans_2020}
Daniel {Glasscock}, Joel {Moreira}, and Florian~K. {Richter}.
\newblock {Additive transversality of fractal sets in the reals and the
  integers}.
\newblock {\em arXiv e-prints},  arXiv:2007.05480, July 2020.

\bibitem[HS]{localentropy}
Michael Hochman and Pablo Shmerkin.
\newblock Local entropy averages and projections of fractal measures.
\newblock {\em Ann. of Math. (2)}, 175(3):1001--1059, 2012.

\bibitem[KM]{Mattila_kaufman_example_1975}
R.~Kaufman and P.~Mattila.
\newblock Hausdorff dimension and exceptional sets of linear transformations.
\newblock {\em Ann. Acad. Sci. Fenn. Ser. A I Math.}, 1(2):387--392, 1975.

\bibitem[KN]{Kuipers_Niederreiter_UDbook}
Lauwerens Kuipers and Harald Niederreiter.
\newblock {\em Uniform distribution of sequences}.
\newblock Wiley-Interscience [John Wiley \& Sons], New York-London-Sydney,
  1974.
\newblock Pure and Applied Mathematics.

\bibitem[KT]{katztao2001}
Nets~Hawk Katz and Terence Tao.
\newblock Some connections between {F}alconer's distance set conjecture and
  sets of {F}urstenburg type.
\newblock {\em New York J. Math.}, 7:149--187, 2001.

\bibitem[LM1]{lima_moreira_comb_proof_of_marstrand}
Yuri Lima and Carlos~Gustavo Moreira.
\newblock A combinatorial proof of {M}arstrand's theorem for products of
  regular {C}antor sets.
\newblock {\em Expo. Math.}, 29(2):231--239, 2011.

\bibitem[LM2]{limamoreira}
Yuri Lima and Carlos~Gustavo Moreira.
\newblock A {M}arstrand theorem for subsets of integers.
\newblock {\em Combinatorics, Probability and Computing}, 23:116--134, 1 2014.

\bibitem[Mar]{Marstrand_1954}
John~M. Marstrand.
\newblock Some fundamental geometrical properties of plane sets of fractional
  dimensions.
\newblock {\em Proc. London Math. Soc. (3)}, 4:257--302, 1954.

\bibitem[Mat]{mattila_geometry_of_sets_1995}
Pertti Mattila.
\newblock {\em Geometry of sets and measures in {E}uclidean spaces}, volume~44
  of {\em Cambridge Studies in Advanced Mathematics}.
\newblock Cambridge University Press, Cambridge, 1995.
\newblock Fractals and rectifiability.

\bibitem[Mor]{moreira_sums_of_regular_cantor_sets}
Carlos Gustavo~T. Moreira.
\newblock Sums of regular {C}antor sets, dynamics and applications to number
  theory.
\newblock In {\em International Conference on Dimension and Dynamics (Miskolc,
  1998)}, volume~37,  55--63. 1998.

\bibitem[Orp]{Orponen_2015}
Tuomas Orponen.
\newblock On the packing dimension and category of exceptional sets of
  orthogonal projections.
\newblock {\em Ann. Mat. Pura Appl. (4)}, 194(3):843--880, 2015.

\bibitem[PS]{peresshmerkin2009}
Yuval Peres and Pablo Shmerkin.
\newblock Resonance between {C}antor sets.
\newblock {\em Ergodic Theory Dynam. Systems}, 29(1):201--221, 2009.

\bibitem[Shm]{shmerkin}
Pablo Shmerkin.
\newblock On {F}urstenberg's intersection conjecture, self-similar measures,
  and the {$L^q$} norms of convolutions.
\newblock {\em Ann. of Math. (2)}, 189(2):319--391, 2019.

\bibitem[Wu]{wu}
Meng Wu.
\newblock A proof of {F}urstenberg's conjecture on the intersections of
  {$\times p$}- and {$\times q$}-invariant sets.
\newblock {\em Ann. of Math. (2)}, 189(3):707--751, 2019.

\bibitem[Yu]{yu_improvement_to_furstenberg_arxiv}
Han Yu.
\newblock An improvement on {F}urstenberg's intersection problem.
\newblock {\em Trans. Amer. Math. Soc.}, 374(9):6583--6610, 2021.

\end{thebibliography}

\bigskip
\footnotesize
\noindent
Daniel Glasscock\\
\textsc{University of Massachusetts Lowell}\par\nopagebreak
\noindent
\href{mailto:daniel_glasscock@uml.edu}
{\texttt{daniel{\_}glasscock@uml.edu}}

\bigskip
\footnotesize
\noindent
Joel Moreira\\
\textsc{University of Warwick}\par\nopagebreak
\noindent
\href{mailto:joel.moreira@warwick.ac.uk}
{\texttt{joel.moreira@warwick.ac.uk}}

\bigskip
\footnotesize
\noindent
Florian K.\ Richter\\
\textsc{Northwestern University}\par\nopagebreak
\noindent
\href{mailto:fkr@northwestern.edu}
{\texttt{fkr@northwestern.edu}}

\end{document}